\documentclass[12pt]{amsart}

\usepackage{amsfonts}
\usepackage{amssymb,}

\catcode`\@=11

 \def\AMSTeXfeatures{\Plainheads 
   \let\current@vert=\AMS@vert}

 \def\Plainheads{\sh@ftdiam=0.05em
   \getlabeldims
   \let\vshaftfill=\plnvsolidfill
   \let\hshaftfill=\plnhsolidfill
   \let\th@rhead=\plnrhead
   \let\th@lhead=\plnlhead
   \let\th@dnhead=\plndnhead
   \let\th@uphead=\plnuphead}
 
 \def\glet{\global\let}

 \def\LaTeXfeatures{\catcode`\@=11
   \ifx\@clnwd\undefined \nol@g
      \input ltxcode.tex \dol@g \fi
   \ltxheads \let\current@vert=\new@vert
   \providelto \catcode`\@=\active}

 \def\nol@g{\def\wlog{\edef\garbage}}
 \def\dol@g{\let\wlog=\wl@g} \let\wl@g=\wlog
 \nol@g 

 \newbox\ltobox
 \def\providelto{{\setbox\z@=
   \hbox{$\to$}\minharrlen=\wd\z@
   \global\setbox\ltobox=\hbox{$\activeat>>>$}}
   \def\lto{\mathrel{\copy\ltobox}}}

 \def\ltxheads{\sh@ftdiam=\@wholewidth
   \getlabeldims
   \let\vshaftfill= \ltxvsolidfill
   \let\hshaftfill=\ltxhsolidfill
   \let\th@rhead=\ltxrhead
   \let\th@lhead=\ltxlhead
   \let\th@dnhead=\ltxdnhead
   \let\th@uphead=\ltxuphead}
 {\catcode`\@=\active
   \gdef@#1{\csname #1\string@at\endcsname}
   \glet\activeat=@}
 \def\def@#1{\expandafter\def\csname #1@at\endcsname}

 \def@>#1>#2>{\@rrow R{#1}{#2}}
 \def@<#1<#2<{\@rrow L{#1}{#2}}
 \def@ V#1V#2V{\@rrow V{#1}{#2}}
 \def@ A#1A#2A{\@rrow A{#1}{#2}}
 \def@/#1/#2/#3/{\@rrow{#1}{#2}{#3}}
 \def@.{\ifodd\row\ifmmode\noharrow
     \else\leavevmode.\spacefactor3000 \fi
   \else\novarrow\fi}
 \def@={\ifodd\row\harrow\hequalfill{}{}%
   \else\varrow\vequalfill{}{}\fi}
 \def@:#1{\ifx=#1\harrow\deffill{}{}%
   \else\leavevmode\null:#1\fi}
 \def@|{\current@vert}
  \def\AMS@vert{\varrow\vequalfill{}{}}
  \def\new@vert#1|#2|{\ifodd\row
   \let\nextarrow\vertexvarrow
   \else\let\nextarrow\varrow\fi
   \nextarrow\vshaftfill{#1}{#2}}
 \def@-{\ifmmode\let\next\hl@ne
   \else\let\next\AMSatdash \fi \next}
  \def\hl@ne#1-#2-{\harrow\hshaftfill{#1}{#2}}
  \def\AMSatdash{\let\next\relax\leavevmode
    \def\next@{\ifx\next-%
      \def\next-{\futurelet\next\nextii@}%
     \else\def\next{\hbox{-}}\fi\next}%
    \def\nextii@{\ifx\next-\def\next-{\hbox{---}}%
      \else\def\next{\hbox{--}}\fi\next}%
    \futurelet\next\next@}
 \def@(#1){\tweenarrows{#1}}
 \def@[#1]{\setsp@n#1\relax\activeat}
 \def\fiberbox{\hbox{$\vcenter{\hr@le\hbox{\vr@le
   \kern1ex\vbox{\kern1.2ex}\vr@le}\hr@le}$}}
  \def\hr@le{\hrule height \sh@ftdiam}
  \def\vr@le{\vrule width \sh@ftdiam}
 \def@+#1+#2+#3+{\ifodd\row \harrow{#1}{#2}{#3}%
   \else \varrow{#1}{#2}{#3}\fi}

 \def\Dnarrfill{\vequalfill\Dnhe@d}
 \def\Uparrfill{\Uphe@d\vequalfill}
 
 \def\ontofill{\rtarrfill\kern-0.3em 
   \th@rhead\kern 0.3em} 

 \def\rtarrfill{\hshaftfill\th@rhead}
 \def\ltarrfill{\th@lhead\hshaftfill}
 \def\dnarrfill{\vshaftfill\th@dnhead}
 \def\uparrfill{\th@uphead\vshaftfill}
 \def\hequalfill{\plnhfill=}
 \def\deffill{:\plnhfill=}
 \def\plnvextfill#1{\setbox\z@
   \hbox{\the\textfont3 #1}%
   \dimen@=\dp\z@\advance\dimen@\ht\z@
   \copy\z@ \kern-\dimen@ 
   \cleaders\copy\z@ \vfill
   \kern-\dimen@ 
   \box\z@}
 \def\plnhfill#1{$\m@th\mkern-1.5mu\mathord#1\mkern-6mu
    \cleaders\hbox{$\mkern-2mu\mathord#1\mkern-2mu$}\hfill
    \mkern-6mu\mathord#1\mkern-1.5mu$}
 \def\vequalfill{\plnvextfill{\char'167}}
 \def\plnvsolidfill{\plnvextfill{\char'077}}
 \def\plnhsolidfill{\plnhfill-}
 \def\ltxhsolidfill{\leaders\hrule height\topofshaft depth\botofshaft
   \hfill}
 \def\ltxvsolidfill{\leaders\vrule width\sh@ftdiam\vfill}
 \def\hdashfill{\hd@sh\wd@sh
   \xleaders \hbox{\wd@sh\hd@sh\wd@sh}\hfill
   \wd@sh\hd@sh}
 \def\vdashfill{\vd@sh\wd@sh
   \xleaders \vbox{\wd@sh\vd@sh\wd@sh}\vfill
   \wd@sh\vd@sh}
 \def\dashed{\ifinmeasureCD\else
    \ifodd\row\option{\let\hshaftfill=\hdashfill}%
   \else\option{\let\vshaftfill=\vdashfill}\fi\fi}

 \newdimen\CDstrutht  \newdimen\CDstrutdp
 \newdimen\CDstrutlen \CDstrutlen=\CDstrutht
   \advance\CDstrutlen by \CDstrutdp

 \def\CDstrut{\vrule
   height \ifnum\row=1 \z@\else\CDstrutht \fi
   depth \ifnum\row=\numrows \z@ \else\CDstrutdp \fi
   width\z@}

 \newdimen\CDarrsurr \CDarrsurr=0.375em
 \newdimen\CDdashlen
 \newdimen\CDvarrlen \CDvarrlen=1.5\baselineskip
 \newdimen\minharrlen 
  \setbox\z@\hbox{$\longrightarrow$} \minharrlen=\wd\z@
 \newdimen\minCDharrlen \minCDharrlen=2.5em 
\newdimen \minc@lwd
\def\findminc@lwd{\minc@lwd=2\CDarrsurr
  \advance\minc@lwd\minCDharrlen}

 \newdimen\sh@ftdiam

 \newdimen\labelsurr \labelsurr=1.25 em

\newcount\sp@ncnt \sp@ncnt=\@ne
\newcount\sp@ncnt@ \sp@ncnt@=\@ne
\newdimen\@rrwd \newdimen\@rrdp

 \def\adjustbot#1{\option{\advance\@rrdp#1\relax}}
 
\def\pushvertex#1{\global\p@shlen#1\relax
   \global\let\maybepush=\dopush}

 \newdimen\p@shlen \p@shlen=\z@

 \let\maybepush=\relax
 \def\dopush{\ifinmeasureCD 
   \advance\locdimen by -\p@shlen 
   \else\advance \@rrwd by -\p@shlen \fi 
   \global\let\maybepush=\relax \global\p@shlen=\z@\relax}

 \def\span@ne{\global\sp@ncnt=\@ne\relax}
 \def\setsp@n#1#2{\global\sp@ncnt=#1\relax
   \ifx\relax#2\relax\else\global\sp@ncnt@=#2\relax\fi}

 \def\plnrhead{\llap{$\rightarrow\mkern-1.5mu$}}
 \def\plnlhead{\rlap{$\mkern-1.5mu\leftarrow$}}

 \def\clap#1{\hbox to \z@{\hss #1\hss}}

 \def\plndnhead{\hbox{\the\textfont3 \char'171}}
 \def\plnuphead{\hbox{\the\textfont3 \char'170}}
 \def\Dnhe@d{\hbox{\the\textfont3 \char'177}}
 \def\Uphe@d{\hbox{\the\textfont3 \char'176}}

 \def\ltxrhead{\raise\@xisheight
   \llap{\smash{\@linefnt\@getrarrow(1,0)}}}
 \def\ltxlhead{\raise\@xisheight
   \rlap{\@linefnt\@getlarrow(-1,0)}}
 \def\ltxuphead{\setbox\z@=\rlap{%
   \kern\@halfwidth\@linefnt\char'66}%
   \copy\z@\kern-\ht\z@}
 \def\ltxdnhead{\setbox\z@=\rlap{%
   \kern\@halfwidth\@linefnt\char'77}%
   \ht\z@=\z@\box\z@}

 \def\wd@sh{\kern0.5\CDdashlen}
 \def\hd@sh{\vrule height\topofshaft depth\botofshaft
    width\CDdashlen}
 \def\vd@sh{\hrule height\CDdashlen
   depth\z@ width\sh@ftdiam}

\def\xylist{14{3434}13{2414}12{1723}%
  23{1413}34{1153}11{0867}43{0707}%
  32{0580}21{0414}31{0291}41{0}}
\newcount\tgtcnt@
\def\find@xyargs{\dimen@=\@rrdp
  \advance\dimen@ by \CDstrutlen
  \tgtcnt@=\dimen@ \dimen@=\@rrwd 
  \divide\dimen@ by \@m 
  \divide \tgtcnt@ by \dimen@ 
  \expandafter\testxy\xylist\relax
  \unitlength=\@xarg\@rrdp
  \divide\unitlength by\@yarg\relax}
\def\testxy#1#2#3{\ifnum\tgtcnt@>#3
    \@xarg=#1\relax \@yarg=#2\relax
    \let\next=\ignorerest
  \else\let\next\testxy\fi\next}
\def\ignorerest#1\relax{\relax}

\let\scalefactor=\@ne
\def\SWarrow{\find@xyargs\vector
  (-\@xarg,-\@yarg)\scalefactor\hskip-\wd\@linechar}
\def\NWarrow{\find@xyargs\vector
  (-\@xarg,\@yarg)\scalefactor\hskip-\wd\@linechar}
\def\NEarrow{\find@xyargs\vector
  (\@xarg,\@yarg)\scalefactor}
\def\SEarrow{\find@xyargs\vector
  (\@xarg,-\@yarg)\scalefactor}
\def\rightupline{\find@xyargs\@linelen=\scalefactor
     \unitlength\@sline}
\def\rightdownline{\find@xyargs\@yarg=-\@yarg\relax
     \@linelen=\scalefactor\unitlength\@sline}

\def\Sim{\ifodd\row\setbox\z@=\hbox{$\sim$}\dimen@=\ht\z@
 \advance\dimen@ by -\@xisheight
  \vbox{\box\z@\kern-\@xisheight\kern\dimen@}%
  \else\hbox{$\wr$}\fi}

\def\harrow#1#2#3{\inmeasureCDtrue\findminarrwd
  {#2}{#3}{\sp@ncnt\minharrlen}\inmeasureCDfalse\span@ne
  \mathrel{\hbox{\options\hplace{#1}\ulabel{#2}\dlabel{#3}}}}

\def\noharrow{\harrow\hfill{}{}}
\def\vertexvarrow#1#2#3{\findarrdp \@rrwd=\z@ \setsp@n\@ne\@ne
  \vbox to \z@{\kern-1.2\CDstrutht
  \rlap{\options\vplace{#1}\llabel{#2}\rlabel{#3}}\vss}}

\newif\ifinmeasureCD
\def\measurelabel#1{\setbox\z@
  \hbox{$\scriptstyle#1\kern\labelsurr$}%
  \ifdim\wd\z@>\@rrwd \@rrwd=\wd\z@\fi}
\def\findminarrwd#1#2#3{\@rrwd=#3\relax
   \measurelabel{#1}\measurelabel{#2}}
\def\findCDarrwd#1#2{\@rrwd=\minCDharrlen
   \measurelabel{#1}\measurelabel{#2}%
  }

\newcount\row \row=\@ne \newcount\col \col=\@ne 
 \newcount\numrows 
\numrows=\@ne
 \newcount\numcols
\newcount\arrspan \newdimen\vrtxhalfwd  \newbox\tempbox

\def\DANABUG{\advance\col by \@ne
 \@rrwd=\minCDharrlen
  \advance\@rrwd by \vrtxhalfwd
  \advance\@rrwd by \CDarrsurr
  \ifnum\col>\numcols \numcols=\col
     \newlocdimen{col\the\col}\locdimen=\@rrwd 
  \else \ifdim\@rrwd>\c@l \c@l=\@rrwd\fi\fi}

\def\drop#1\\{
  \findvrtxhalfsum\DANABUG\advance\row by 2 \measureinit}

\def\measureinit{\col=\@ne \vrtxhalfwd=-\CDarrsurr\arrspan=\@ne\@rrwd=\z@
   \setbox\tempbox=\hbox\bgroup$}
\def\measure{
  \let\harrow\measureCDarrow
  \let\CDCR=\measureCR 
   \findminc@lwd 
  \inmeasureCDtrue
  \row=\@ne \numcols=\z@ \measureinit}

\def\endmeasure{\findvrtxhalfsum\DANABUG
  \numrows=\row 
  \inmeasureCDfalse}

\def\newlocdimen#1{\advance\dimenc@unt by \@ne
  \ifnum\dimenc@unt<\insc@unt
     \else\errmessage{No room for the CD}\fi
  \dimendef\locdimen=\dimenc@unt
  \expandafter\dimendef\csname#1\endcsname=\dimenc@unt}

 \def\r@wc@l{\csname row\the\row col\the\col\endcsname}
 \def\c@l{\csname col\the\col\endcsname}

 \def\findvrtxhalfsum{$\egroup
  \newlocdimen{row\the\row col\the\col}
  \locdimen=\vrtxhalfwd 
  \vrtxhalfwd=0.5\wd\tempbox 
  \advance\vrtxhalfwd by \CDarrsurr
  \advance\locdimen by \vrtxhalfwd 
  \advance\@rrwd by \locdimen 
  \maybepush
  \divide\@rrwd by \arrspan\relax
  \ifdim\@rrwd<\minc@lwd
    \ifnum\col>\@ne \@rrwd=\minc@lwd\fi \fi
  \loop 
    \ifnum\col>\numcols \numcols=\col
       \newlocdimen{col\the\col}
       \locdimen=\@rrwd 
    \else \ifdim\@rrwd>\c@l \c@l=\@rrwd\fi \fi
   \ifnum\arrspan>\@ne
      \advance\arrspan by -1 \advance\col by \@ne
  \repeat }

 \def\measureCDarrow#1#2#3{\findvrtxhalfsum
   \arrspan=\sp@ncnt\relax\global\sp@ncnt=1\relax
   \advance\col by \@ne
   \findCDarrwd{#2}{#3}%
   \setbox\tempbox=\hbox\bgroup$}

 \newcount\dr@tn \dr@tn=\z@
 \def\locate#1:#2{\ifinmeasureCD\else
   \count@=-#1
   \multiply\count@ by 2
   \advance\count@ by #2
   \dimen@=\count@\@rrwd
   \ifnum\dr@tn=\@ne\relax \else\dimen@=-\dimen@ \fi
   \dimen@i=\@rrdp
   \ifnum\dr@tn>\z@\advance\dimen@i by \CDstrutlen \fi
   \dimen@i=\count@\dimen@i
   \count@=#2 \multiply\count@ by 2
   \divide\dimen@ by \count@
   \divide\dimen@i by \count@
   \lift\dimen@i\nudge\dimen@\fi}

\def\betweenCDrows{\advance\row by \@ne \col=\@ne
\options}

\def\hbegin{\hbox\bgroup\kern\c@l \kern-\r@wc@l$}
\def\hend{$\glet\maybepush\relax \CDstrut\egroup}
\def\vbegin{\setbox\tempbox=\hbox\bgroup$}
\def\vend{$\egroup\ht\tempbox=\z@\dp\tempbox\CDvarrlen
  \box\tempbox}
\def\setCD{\let\harrow=\setCDarrow
  \let\CDCR=\setCR 
  \row=\@ne \col=\@ne \hbegin}
\let\endsetCD=\hend 

\def\findarrwd{\@rrwd=\z@ \count@=\col \advance\count@ by\sp@ncnt
  \loop\ifnum\count@>\col \advance\count@ by -1
      \advance\@rrwd by\csname col\the\count@\endcsname\repeat}
\def\setCDarrow#1#2#3{\kern\CDarrsurr\advance\col by \@ne
  \findarrwd \advance\@rrwd by -\r@wc@l  
  \@rrdp=\z@ 
  \maybepush
  \advance\col by -\@ne \advance\col by \sp@ncnt \span@ne
  \hbox to \@rrwd{\options
   \@rrwd=\scalefactor\@rrwd\hss
   \hplace{#1}\ulabel{#2}\dlabel{#3}\hss}%
   \kern\CDarrsurr}

\newdimen\labspacei 
\newdimen\labspaceii 

\newdimen\@xisheight
  \@xisheight=\the\fontdimen22\textfont2
\newdimen\labelskip
  \labelskip=\the\fontdimen10\textfont3 
\newdimen\topofshaft
\newdimen\botofshaft
\newdimen\botofulabel
\newdimen\topofdlabel
\def\getlabeldims{
  \topofshaft=0.5\sh@ftdiam
  \botofshaft=\topofshaft
  \advance\topofshaft by \@xisheight  
  \advance\botofshaft by -\@xisheight  
  \botofulabel=\topofshaft
  \advance\botofulabel by \labelskip
  \topofdlabel=\botofshaft
  \advance\topofdlabel by \labelskip}

\def\ulabel{\ifnum\row=\@ne\let\next\ulabeli
   \else\let\next\ulabellap\fi\next}
\def\ulabeli#1{\vbox{
  \clap{\kern-\@rrwd$\scriptstyle#1$}%
  \kern\botofulabel}\maybeoffset}
\def\ulabellap#1{\vbox to \z@{\vss
  \clap{\kern-\@rrwd$\scriptstyle#1$}%
  \kern\botofulabel}\maybeoffset}
\def\dlabel{\ifnum\row=\numrows\let\next\dlabeli
   \else\let\next\dlabellap\fi\next}
\def\dlabeli#1{\vtop{\kern\topofdlabel
  \clap{\kern-\@rrwd$\scriptstyle#1$}%
  }\maybeoffset}
\def\dlabellap#1{\vbox to \z@{\kern\topofdlabel
  \clap{\kern-\@rrwd$\scriptstyle#1$}%
  \vss}\maybeoffset}
\def\rlabel#1{\vbox to \z@{\vss
  \rlap{\kern\labelskip$\scriptstyle#1$}%
  \vss\kern-\@rrdp}\maybeoffset}
\def\llabel#1{\vbox to \z@{\vss
  \llap{$\scriptstyle#1$\kern\labelskip}%
  \vss\kern-\@rrdp}\maybeoffset}
\def\swlabel#1{\vtop{\kern0.5\@rrdp
  \llap{$\scriptstyle#1$\kern\labelskip\kern-0.5\@rrwd}
  }\maybeoffset}
\def\nwlabel#1{\vbox{
  \llap{$\scriptstyle#1$\kern\labelskip\kern-0.5\@rrwd}%
  \kern-0.5\@rrdp}\maybeoffset}
\def\selabel#1{\vtop{\kern0.5\@rrdp
  \rlap{\kern0.5\@rrwd\kern\labelskip$\scriptstyle#1$}%
  }\maybeoffset}
\def\nelabel#1{\vbox{
  \rlap{\kern0.5\@rrwd\kern\labelskip$\scriptstyle#1$}%
  \kern-0.5\@rrdp}\maybeoffset}
\def\cplace#1{\vbox to \z@{\vss
  \clap{$#1$\kern-\@rrwd}%
  \kern-\@rrdp\vss}\maybeoffset}
\def\hplace#1{\hbox to \@rrwd{#1}\maybeoffset}
\def\vplace#1{\clap{\vbox to \z@{#1\kern-\@rrdp}}\maybeoffset}

\newdimen\nudgeamount \nudgeamount=\z@
\newdimen\liftamount \liftamount=\z@
\let\maybeoffset\relax
\newbox\offsetbox \newdimen\lastheight
\def\dooffset{
  \setbox\offsetbox=\lastbox \lastheight=\ht\offsetbox 
  \setbox\offsetbox=\vbox{\kern-\liftamount\box\offsetbox}%
  \ht\offsetbox=\lastheight
  \kern\nudgeamount\box\offsetbox\kern-\nudgeamount
  \global\nudgeamount=\z@ \global\liftamount=\z@
  \glet\maybeoffset=\relax}
\def\nudge#1{\ifinmeasureCD\else
  \global\advance\nudgeamount#1\relax
  \global\let\maybeoffset\dooffset\fi}
\def\lift#1{\ifinmeasureCD\else
  \global\advance\liftamount#1\relax
  \global\let\maybeoffset\dooffset\fi}

\def\findarrdp{\@rrdp=\CDvarrlen
  \ifnum\sp@ncnt@>1
    \advance\@rrdp by \CDstrutlen
    \multiply\@rrdp by \sp@ncnt@
    \advance\@rrdp by -\CDstrutlen \fi
 }

\def\varrow#1#2#3{\ifnum\sp@ncnt>\@ne 
     \sp@ncnt@=\sp@ncnt\relax\fi
  \findarrdp \@rrwd=\z@ 
  \kern\c@l
   \hbox to \z@{\options
   \@rrdp=\scalefactor\@rrdp
    \hss\vplace{#1}\llabel{#2}\rlabel{#3}\hss}%
  \global\advance\col by \@ne \setsp@n\@ne\@ne
  }

\def\novarrow{\varrow\vfill{}{}}

\def\tweenarrows#1{\findarrwd \findarrdp \setsp@n\@ne\@ne
  \rlap{\options\cplace{#1}}}

\def\usarrow #1#2#3{\dr@tn=\@ne
  \findarrwd \findarrdp \setsp@n\@ne\@ne 
  \rlap{\options\cplace{#1}\nwlabel{#2}\selabel{#3}}%
  \dr@tn=\z@}
\def\dsarrow #1#2#3{\dr@tn=\tw@
  \findarrwd \findarrdp \setsp@n\@ne\@ne 
  \rlap{\options\cplace{#1}\swlabel{#2}\nelabel{#3}}%
  \dr@tn=\z@}
 \def\@rrow#1{\csname #1@rrow\endcsname}
 \def\R@rrow{\harrow \rtarrfill}
 \def\L@rrow{\harrow \ltarrfill}
 \def\V@rrow{\varrow \dnarrfill}
 \def\A@rrow{\varrow \uparrfill}
 \def\SE@rrow{\dsarrow \SEarrow}
 \def\NW@rrow{\dsarrow \NWarrow}
 \def\SW@rrow{\usarrow \SWarrow}
 \def\NE@rrow{\usarrow \NEarrow}
 \def\DS@rrow{\dsarrow \dnslope}
 \def\US@rrow{\usarrow \upslope}
 \def\upslope{\find@xyargs
       \@linelen=\unitlength\@sline}
 \def\dnslope{\find@xyargs\@yarg=-\@yarg\relax
       \@linelen=\unitlength\@sline}

\newtoks\optionlist 
\optionlist={}
\let\options\relax
\def\dooptions{\the\optionlist\global\optionlist={}%
  \glet\options=\relax}
\def\option#1{\ifinmeasureCD\else
  \glet\options=\dooptions
  \global\optionlist=\expandafter{\the\optionlist\relax#1}\fi}
\def\wider#1{\ifinmeasureCD\else
   \option{\advance\@rrwd by #1}\fi}
\def\deeper#1{\ifinmeasureCD\else
   \option{\advance\@rrdp by #1}\fi}

{\def\\{\global\let\sptoken= }\\ }

\def\CR{\futurelet\nexttok\testCR}
\def\testCR{\ifx\nexttok\sptoken
   \let\next\eatspaceCR\else\let\next\CDCR\fi\next}
\def\eatspaceCR#1 {\CR}
\def\measureCR{\ifx\nexttok\endmeasure\let\nextCR\relax
    \else\let\nextCR\drop\fi\nextCR}
\def\setCR{\ifodd\row
  \ifx\nexttok\endsetCD\else\hend\betweenCDrows\vbegin\fi
  \else\vend\betweenCDrows\hbegin\fi}

\countdef\dimenc@unt=11
\def\CD#1\endCD{
   \begingroup\let\\=\CR
  \m@th\offinterlineskip
   \measure#1\endmeasure\null\,\vcenter{\setCD#1\endsetCD}\,
   \endgroup
    }

\ifx\@clnwd\undefined \nol@g\else\catcode`\ =14\relax\fi
 \font\@linefnt=line10 
 \newcount\@tempcnta
 \newcount\@tempcntb
 \newdimen\@tempdima
 \newdimen\@tempdimb
 \newdimen\@wholewidth
 \newdimen\@halfwidth
   \@wholewidth\fontdimen8\@linefnt \@halfwidth .5\@wholewidth
 \newdimen\unitlength
 \newcount\@xarg
 \newcount\@yarg
 \newcount\@yyarg
 \newbox\@linechar
 \newdimen\@linelen
 \newdimen\@clnwd
 \newdimen\@clnht
 \newif\if@negarg
 
 \def\@whilenoop#1{}

 \def\@whiledim#1\do #2{\ifdim #1\relax#2\@iwhiledim{#1\relax#2}\fi}

 \def\@iwhiledim#1{\ifdim #1\let\@nextwhile=\@iwhiledim 
         \else\let\@nextwhile=\@whilenoop\fi\@nextwhile{#1}}

 \def\@sline{\ifnum\@xarg< 0 \@negargtrue \@xarg -\@xarg \@yyarg -\@yarg
   \else \@negargfalse \@yyarg \@yarg \fi
 \ifnum \@yyarg >0 \@tempcnta\@yyarg \else \@tempcnta -\@yyarg \fi
 \ifnum\@tempcnta>6 \@badlinearg\@tempcnta0 \fi
 \ifnum\@xarg>6 \@badlinearg\@xarg 1 \fi
 \setbox\@linechar\hbox{\@linefnt\@getlinechar(\@xarg,\@yyarg)}%
 \ifnum \@yarg >0 \let\@upordown\raise \@clnht\z@
    \else\let\@upordown\lower \@clnht \ht\@linechar\fi
 \@clnwd=\wd\@linechar
 \if@negarg \hskip -\wd\@linechar \def\@tempa{\hskip -2\wd\@linechar}\else
      \let\@tempa\relax \fi
 \@whiledim \@clnwd <\@linelen \do
   {\@upordown\@clnht\copy\@linechar
    \@tempa
    \advance\@clnht \ht\@linechar
    \advance\@clnwd \wd\@linechar}%
 \advance\@clnht -\ht\@linechar
 \advance\@clnwd -\wd\@linechar
 \@tempdima\@linelen\advance\@tempdima -\@clnwd
 \@tempdimb\@tempdima\advance\@tempdimb -\wd\@linechar
 \if@negarg \hskip -\@tempdimb \else \hskip \@tempdimb \fi
 \multiply\@tempdima \@m
 \@tempcnta \@tempdima \@tempdima \wd\@linechar \divide\@tempcnta \@tempdima
 \@tempdima \ht\@linechar \multiply\@tempdima \@tempcnta
 \divide\@tempdima \@m
 \advance\@clnht \@tempdima
 \ifdim \@linelen <\wd\@linechar
    \hskip \wd\@linechar
   \else\@upordown\@clnht\copy\@linechar\fi}
 
 \def\@getlinechar(#1,#2){\@tempcnta#1\relax\multiply\@tempcnta 8
 \advance\@tempcnta -9 \ifnum #2>0 \advance\@tempcnta #2\relax\else
 \advance\@tempcnta -#2\relax\advance\@tempcnta 64 \fi
 \char\@tempcnta}
 
 \def\vector(#1,#2)#3{\@xarg #1\relax \@yarg #2\relax
 \@tempcnta \ifnum\@xarg<0 -\@xarg\else\@xarg\fi
 \ifnum\@tempcnta<5\relax
 \@linelen=#3\unitlength
 \ifnum\@xarg =0 \@vvector 
   \else \ifnum\@yarg =0 \@hvector \else \@svector\fi
 \fi
 \else\@badlinearg\fi}
 
 \def\@svector{\@sline
 \@tempcnta\@yarg \ifnum\@tempcnta <0 \@tempcnta=-\@tempcnta\fi
 \ifnum\@tempcnta <5
   \hskip -\wd\@linechar
   \@upordown\@clnht \hbox{\@linefnt  \if@negarg 
   \@getlarrow(\@xarg,\@yyarg) \else \@getrarrow(\@xarg,\@yyarg) \fi}%
 \else\@badlinearg\fi}
 
 \def\@getlarrow(#1,#2){\ifnum #2 =\z@ \@tempcnta='33\else
 \@tempcnta=#1\relax\multiply\@tempcnta \sixt@@n \advance\@tempcnta
 -9 \@tempcntb=#2\relax\multiply\@tempcntb \tw@
 \ifnum \@tempcntb >0 \advance\@tempcnta \@tempcntb\relax
 \else\advance\@tempcnta -\@tempcntb\advance\@tempcnta 64
 \fi\fi\char\@tempcnta}
 
 \def\@getrarrow(#1,#2){\@tempcntb=#2\relax
 \ifnum\@tempcntb < 0 \@tempcntb=-\@tempcntb\relax\fi
 \ifcase \@tempcntb\relax \@tempcnta='55 \or 
 \ifnum #1<3 \@tempcnta=#1\relax\multiply\@tempcnta
 24 \advance\@tempcnta -6 \else \ifnum #1=3 \@tempcnta=49
 \else\@tempcnta=58 \fi\fi\or 
 \ifnum #1<3 \@tempcnta=#1\relax\multiply\@tempcnta
 24 \advance\@tempcnta -3 \else \@tempcnta=51\fi\or 
 \@tempcnta=#1\relax\multiply\@tempcnta
 \sixt@@n \advance\@tempcnta -\tw@ \else
 \@tempcnta=#1\relax\multiply\@tempcnta
 \sixt@@n \advance\@tempcnta 7 \fi\ifnum #2<0 \advance\@tempcnta 64 \fi
 \char\@tempcnta}
\catcode`\ =10

\dol@g 
\catcode`\@=\active
\LaTeXfeatures

\begin{document}

\def\al{Val}
\numberwithin{equation}{section}

\newtheorem{theorem}{Theorem}[section]
\newtheorem{lemma}[theorem]{Lemma}

\newtheorem*{theorema}{Theorem A}
\newtheorem*{theorema1}{Theorem A${}^\prime$}
\newtheorem*{theoremb}{Theorem B}
\newtheorem*{corc}{Corollary C}

\newtheorem{prop}[theorem]{Proposition}
\newtheorem{proposition}[theorem]{Proposition}
\newtheorem{corollary}[theorem]{Corollary}
\newtheorem{corol}[theorem]{Corollary}
\newtheorem{conj}[theorem]{Conjecture}
\newtheorem{sublemma}[theorem]{Sublemma}
\newtheorem{quest}[theorem]{Question}

\theoremstyle{definition}
\newtheorem{defn}[theorem]{Definition}
\newtheorem{example}[theorem]{Example}
\newtheorem{examples}[theorem]{Examples}
\newtheorem{remarks}[theorem]{Remarks}
\newtheorem{remark}[theorem]{Remark}
\newtheorem{algorithm}[theorem]{Algorithm}
\newtheorem{question}[theorem]{Question}
\newtheorem{subsec}[theorem]{}
\newtheorem{clai}[theorem]{Claim}
\newtheorem{problem}{Problem}

\renewcommand*{\theproblem}{\arabic{problem}}

\def\toeq{{\stackrel{\sim}{\longrightarrow}}}
\def\into{{\hookrightarrow}}

\def\wt{\widetilde}

\def\kp{Val}


%

\title  [Multi-sorted logic and logical geometry: some problems] {Multi-sorted logic and logical geometry: some problems}

\author[ B. Plotkin] {\bf{ B. Plotkin}}

\address{Boris Plotkin: Institute of
Mathematics, Hebrew University, 91904, Jerusalem, ISRAEL}
\email{plotkin {\it at} macs.biu.ac.il}

\author[E. Plotkin]{E. Plotkin}
\address{Plotkin: Department of Mathematics
 Bar-Ilan University, 52900 Ramat Gan, Israel}



\maketitle

\tableofcontents
\section{Introduction}


\long\def\symbolfootnote[#1]#2{\begingroup%
\def\thefootnote{\fnsymbol{footnote}}\footnote[#1]{#2}\endgroup}

The paper has a form of a survey talk on the given topic.
This second paper continues the first one \cite{Plotkin_Gagta}. It consists of three parts, ordered in a way different from that of \cite{Plotkin_Gagta}. The accents are also different. This paper is focused on the relationship between the many-sorted theory, which  leads to logical geometry and one-sorted theory, which is based on the important model-theoretic concepts. Our aim is to show that both approaches go in parallel and there are bridges which allow to transfer results, notions and problems back and forth. Thus, an additional freedom in choosing an approach appear.

On our opinion, some simple proofs make the paper more vital.

The first part of the paper contains main notions, the second one
is devoted to logical geometry, the third part describes types and
isotypeness. The problems are distributed in the corresponding
parts. The whole material is oriented towards universal algebraic
geometry (UAG), i.e., geometry in an arbitrary variety of algebras
$\Theta$. We will distinguish between the equational algebraic
geometry and the logical geometry. In the equational geometry
equations have the form $w\equiv w'$, where $w$ and $w'$ are
elements of the free in $\Theta$ algebra $W(X)$. In the logical
geometry the elements of the multi-sorted first-order logic  play
the role of equations. We consider
 logical geometry (LG) as a part of UAG. This theory is strongly influenced by model theory and ideas of A.Tarski and A.I.Malcev.

I remember that A.I. Malcev, founding  the journal "Algebra and logic"  in Novosibirsk, had in mind a natural interrelation of these topics.


We fix a variety of algebras $\Theta$. Let $W=W(X)$ be the free in $\Theta$ algebra over a set of variables $X$. The set $X$ is assumed to be finite, if the opposite is not stated explicitly. In the latter case we use the notation $X^0$. All algebras under consideration are algebras in $\Theta$. Logic is also related to the variety $\Theta$. As usual, the signature of  $\Theta$ may contain constants.

\section{Main notions}\label{sec:mn}

In this section we consider a system of notions, we are dealing with. Some of them are not formally defined in this paper.
For the precise definitions and references  use  \cite{Plotkin_UA-AL-Datab}, \cite{Halmos},  \cite{Plotkin_Haz}, \cite{Pl-St}, \cite{PlAlPl}, \cite{MR}.

 The general picture of relations between these notions brings forward a lot of new problems, formulated in the following two sections. These problems are the main objective of the paper. Some results are also presented.

\subsection{Equations, points,  spaces of points and algebra of formulas $\Phi(X)$}\label{sub:1}

Consider a system $T$ of equations of the form $w=w'$, $w, w' \in W(X)$.

Each system $T$ determines an algebraic set of points in the corresponding affine space over the algebra $H \in \Theta$ for every $H$ and every finite $X$.

Let $X=\{x_1, \ldots , x_n \}$. We have an affine space $H^X$ of points  $\mu : X \to H$. For every $\mu$ we have also the $n$-tuple $(a_1, \ldots , a_n) = \bar a$ with $a_i = \mu(x_i)$. For the given $\Theta$ we have the homomorphism $$\mu : W(X) \to H$$ and, hence, the affine space is viewed as the set of homomorphisms $$Hom(W(X),H).$$

The classical kernel $Ker(\mu)$ corresponds to each point $\mu : W(X) \to H$.


Every point $\mu$ has also the logical kernel $LKer(\mu)$. Along with the algebra $W(X)$ we will consider
the algebra of formulas $\Phi(X)$. Logical kernel $LKer(\mu)$ consists of all formulas $u \in \Phi(X)$ valid on the point $\mu$.

The algebra $\Phi(X)$ will be defined later on, but let us note now that it is an extended Boolean algebra (Boolean algebra in which quantifiers $\exists x, x \in X$ act as operations, and equalities ($\Theta$-equalities) $w \equiv w'$, $w, w' \in W(X)$ are defined). It is also defined what does it mean that the point $\mu$ satisfies a formula $u \in \Phi(X)$. These $u$ are treated as equations. For $T \subset \Phi(X)$ in $Hom(W(X),H)$ we have an elementary set (definable set) consisting of points $\mu$ which satisfy every $u \in T$.

Each kernel $LKer(\mu)$ is a Boolean ultrafilter in $\Phi(X)$. Note that
 $$
 Ker(\mu)=LKer(\mu)\cap M_X,
 $$
where $M_X$ is the  set of all  $w\equiv w'$, $w,w'\in W(X)$.

\subsection{Extended Boolean algebras}

Let us make some comments regarding the definition of the notion of extended Boolean algebra.

Let $B$ be a Boolean algebra. An existential quantifier on $B$ is an unary operation $\exists : B\to B$ subject to conditions

\begin{enumerate}
\item
  $\exists (0) =0$,
\smallskip
\item
  $a \le \exists (a) $,
\smallskip
\item
 $\exists (a \wedge \exists b) = \exists a \wedge \exists b$.
\end{enumerate}

\noindent The {\it universal} quantifier $\forall : B \to B$ is defined dually:
\begin{enumerate}
\item
  $\forall(1) =1$,
\smallskip
\item
  $a \ge \forall (a)$,
\smallskip
\item
 $\forall (a \vee \forall b) = \forall a \vee \forall b$.
\end{enumerate}

\noindent Here the numerals $0$ and $1$ are zero and unit of the
Boolean algebra $B$ and $a,b$ are arbitrary elements of $B$.

As usual, the
quantifiers $\exists$ and $\forall$ are coordinated by: $\neg
(\exists a)=\forall (\neg a)$, and $ (\forall a)=\neg(\exists
(\neg a))$.

Now suppose that a variety of algebras $\Theta$ is fixed and $W(X)$ is the free in $\Theta$ algebra over the set of variables $X$. These data allow to define the extended Boolean algebra. This is a Boolean  algebra where the quantifiers $\exists x$ are defined for every $x\in X$  and
$$
\exists x\exists y= \exists y \exists x,
$$
for every $x$ and $y$ from X. Besides that, for every pair of elements $w,w'\in W(X)$ in an extended Boolean algebra the equality $w\equiv w'$ is defined. These equalities  are considered as nullary operations, that is as constants.  Each equality satisfies conditions of an equivalence relation, and  for every operation $\omega$ from the signature of algebras from $\Theta$ we have
$$
w_1\equiv w'_1 \wedge \ldots \wedge w_n \equiv w'_n\to w_1 \ldots
w_n \omega \equiv w'_1 \ldots w'_n \omega.
$$

 Note that quantifiers and Boolean connectives satisfy
  $$
 \exists (a\vee b)=\exists a \vee \exists b
 $$
   $$\forall (a\wedge b)=\forall a \wedge \forall b
 $$

 Algebra of formulas $\Phi(X)$ is an example of extended Boolean algebra in $\Theta$. Now consider another example.

\subsection{Important example}\label{ex:im}

 Let us start from an affine space $Hom(W(X), H)$. Let $Bool(W(X), H)$ be a Boolean algebra of all subsets in $Hom(W(X), H)$. Extend this algebra by adding quantifiers $\exists x$ and equalities. For $A \in Bool(W(X),H)$ we set: $B = \exists x A$ is a set of points $\mu : W(X) \to H$ such that there is $\nu : W(X) \to H$ in $A$ and $\mu(x') = \nu(x')$ for $x' \in X$, $x' \neq x$. It is indeed an  existential quantifier for every $x \in X$.

Define an equality in $Bool(W(X),H)$ for $w \equiv w'$ in $M_X$. Denote it by $[w \equiv w']_H$ and define it, setting $\mu \in [w \equiv w']_H$ if $(w,w') \in Ker(\mu)$, i.e., $w^\mu = {w'}^\mu$.

\begin{remark} The set  $[w \equiv w']_H$ can be empty. Thus we give the following definition. The equality  $[w \equiv w']_H$ is called admissible  for the given $\Theta$ if for every $H\in \Theta$ the set  $[w \equiv w']_H$ is not empty.   If $\Theta$ is the variety of all groups, then each equality is admissible.  The same is true for the variety of associative algebras with unity over complex numbers. However, for the field of real numbers this is not the case. Here $x^2+1=0$ is not an admissible equality.
\end{remark}

We assume that in each algebra of formulas $\Phi(X)$ lie all $\Theta$-equalities. To arbitrary equality
$w \equiv w'$ corresponds either a non-empty equality $[w \equiv w']_H$ in  $H\in \Theta$, or the empty set in $H\in \Theta$ which is the zero element of this Boolean algebra.

We arrived to an extended Boolean algebra, denoted now by $Hal^X_\Theta(H)$. This algebra and the algebra of formulas $\Phi(X)$ have the same signature.

\subsection{Homomorphism $Val^X_H$}\label{sub:qq}\label{sub:hom}

We will proceed from the homomorphism $$Val^X_H : \Phi(X) \to Hal^X_\Theta(H),$$ with the condition $Val^X_H(w \equiv w') = [w \equiv w']_H$ for equalities, if $[w \equiv w']_H$ non-empty, or $0$ otherwise. This homomorphism will be defined in subsection \ref{sub:val}. 
The existence of such homomorphism is not a trivial fact, since the equalities $M_X$ does not generate (and of course does not generate freely) the algebra $\Phi(X)$.
If, further, $u \in \Phi(X)$, then $Val^X_H(u)$ is a set of points in the affine space $Hom(W(X),H)$. We say that a point $\mu$ satisfies the formula $u$ if $\mu$ belongs to $Val^X_H(u)$. Thus, $Val^X_H(u)$ is precisely the set of points satisfying the formula $u$. Define the logical kernel $LKer(\mu)$ of a point $\mu$ as a set of all formulas $u$  such that $\mu \in Val^X_H(u)$.


We have also
$$Ker(\mu) = LKer(\mu) \cap M_X. $$

Here $Ker (\mu)$ is the set of all formulas of the form $w=w'$, $w, w'\in W(X)$, such that the point $\mu$ satisfies these formulas. In parallel, $LKer (\mu)$ is the set of all formulas $u$, such that the point $\mu$ satisfies these formulas.


Then,
$$Ker(Val^X_H) = Th^X(H), $$
$$\bigcap_{\mu:W(X) \to H} LKer(\mu) = Th^X(H). $$
Here $Th^X(H)$ is a set of formulas $u \in \Phi(X)$, such that $Val^X_H(u)$ is  unit in $Bool(W(X),H)$. That is $Val^X_H(u)=Hom(W(X),H)$  and thus $Th^X(H)$ is an $X$-component of the elementary theory of the algebra $H$.

In general we have a multi-sorted representation of the elementary theory $$Th(H)=(Th^X(H), X \in \Gamma).$$

 It follows from the previous considerations that the algebra of formulas $\Phi(X)$ can be embedded 
 in $Hal^X_\Theta(H)$ modulo elementary theory of the algebra $H$. This fact will be used in the sequel.

\subsection{Multi-sorted logic: first approximation}

Let, further, $X^0$ be an infinite set of variables and $\Gamma$ a system of all finite subsets $X$ in $X^0$.

So, in the logic under consideration we have an infinite system $\Gamma$ of finite sets instead of one infinite $X^0$.  This leads to multi-sorted logic. This approach is caused by relations with UAG. We distinguish UAG, equational UAG and LG in UAG. Correspondingly, we have algebraic sets of points and definable sets of points in the affine space.
In the third part of the paper along with the system of sorts $\Gamma$ we use also a system of sorts $\Gamma$ where one initial infinite set $X^0$ is added to the system $\Gamma$.

\subsection{Algebra $Hal_\Theta(H)$}\label{sub:hal}

All these algebras and corresponding categories present universal
semantics for the logic concerning with a variety $\Theta$. Syntax
of this logic is given by the algebra $\tilde \Phi$. The
homomorphism
$$
Val_{H}:\tilde \Phi \to Hal_{\Theta}(H)
$$
gives the correspondence between syntax and semantics. This
homomorphism and the homomorphism
$$
Val_{H}^{X}:\tilde \Phi(X) \to Hal^{X}_{\Theta}(H)
$$
will be defined at the end of the section.

 We
start with the category $\Theta ^{*}(H)$ of affine spaces. Its
objects are spaces
 $Hom(W(X),H)$, where $X \in \Gamma$.

 Morphisms
\[
\wt s: Hom(W(X), H) \to Hom (W(Y), H)
\]
of  $\Theta^{*}(H)$  are mappings induced by homomorphisms $s:W(Y)
\to W(X)$ according to the rule $\wt s(\nu) = \nu s$ for every
$\nu: W(X) \to H$.



 The correspondence $W(X)\to Hom(W(X),H)$ and $s\to \wt s$ gives rise to  a contravariant functor
\[
\varphi  : \Theta^0\to \Theta^{*} (H).
\]

Morphisms $\tilde s$ and $s$ act in the opposite direction. Note that if $s$ is surjective
then $\tilde s$ is injective, and if $s$ is injective then $\tilde s$ is surjective.

\begin{prop}\label{prop:dual}
The functor $\Theta^{0}\to \Theta^{*}(H)$ defines a duality of
categories if and only if $Var(H)=\Theta$.
\end{prop}

\proof
 The condition of duality implies that if $s_1\ne s_2$ for the given morphisms $s_1,s_2: W(Y)\to
 W(X)$ then
 $\widetilde{s_1}\ne \widetilde{s_2}$.

 Let us assume that $Var(H)=\Theta$ and the categories are not dual, so there are
 morphisms $s_1$ and $s_2$ such that  $s_1\ne s_2$ but $\wt{s_1}= \wt{s_2}$.
 Take some $y\in Y$ such that  $s_1(y)=w_1$, $s_2(y)=w_2$ and  $w_1\ne w_2$. We will show
 that in the algebra $H$ there is the non-trivial identity $w_1\equiv
 w_2$. Take an arbitrary homomorphism $\nu: W(X)\to H$. The equality $\tilde s_1=\tilde s_2$
 implies $\tilde s_1(\nu)=\tilde s_2(\nu)$ or $\nu s_1=\nu s_2$.
 We apply this morphism to the variable $y$:
 $$
  \nu s_1(y)=\nu s_2(y)\mbox{ or }  \nu w_1=\nu w_2.
 $$
 Since $\nu:W(X)\to H$ is an arbitrary homomorphism, then $w_1\equiv w_2$ is an identity of the
 algebra $H$. But $Var(H)=\Theta$, which means that  there are no non-trivial identities in $H$. We have a contradiction and the condition $Var(H)=\Theta$
 implies duality of the given categories. 

 Now we show that if $Var(H)\subset \Theta$, then there is no
 duality. Let $w_1\equiv w_2$ be some non-trivial identity of the
 algebra $H$. Take $Y=\{ y_0 \}$ and let $s_1(y_0)=w_1$,
 $s_2(y_0)=w_2$. For any $\nu:W(X)\to H$ we have
 $$
  \nu w_1=\nu w_2, \ \nu s_1(y_0)=\nu s_2(y_0),\ \wt s_1(\nu)(y_0)=\wt s_2(\nu)(y_0).
 $$
Since the set $Y$ contains only one element $y_0$, then $\wt
  s_1(\nu)=\wt s_2(\nu)$. As $\nu $ is arbitrary, then $\wt s_1=\wt s_2$
  and there is no duality of the categories.

  $\square$



Define further a category of all $Hal_\Theta(H)$. Its objects are algebras $Hal^X_\Theta(H)$. Proceed from $s : W(X) \to W(Y)$ and pass to $\tilde s : Hom(W(Y),H) \to Hom(W(X),H)$. Take $$s_\ast(A) = \tilde s ^{-1}(A) = B \subset Hom(W(Y),H)$$ for every object $A \subset Hom(W(X),H)$. We have $\mu \in B$ if and only if $\mu s = \tilde s (\mu) \in A$. This determines a morphism $$s_\ast=s_\ast^H : Hal^X_\Theta(H)\to Hal^Y_\Theta(H).$$ Here $s_\ast$ is well coordinated with the Boolean structure, and relations with quantifiers and equalities are coordinated by  identities  from Definition \ref{def:ha}. The category $Hal_\Theta(H)$ can be also treated as a multi-sorted algebra $$Hal_\Theta(H)=(Hal^X_\Theta(H), X \in \Gamma).$$


\subsection{Variety of Halmos algebras $Hal_\Theta$}
Algebras in $Hal_\Theta$ have the form $$\mathfrak{L}= (\mathfrak{L}_X, \ X\in \Gamma).$$
Here all domains $\mathfrak{L}_X$ are $X$-extended Boolean algebras. The unary operation
$$
 s_\ast: \mathfrak{L}_X \to \mathfrak{L}_Y
 $$
 corresponds to each homomorphism  $
 s: W(X) \to W(Y)$.
 Besides, we will define a category $\mathfrak L$ of all $\mathfrak L_X$, $X \in \Gamma$ with morphisms $s_\ast : \mathfrak L_X \to \mathfrak L_Y$. The transition $s \to s_\ast$ determines a covariant functor $\Theta^0 \to \mathfrak L$. Informally, operations of $s_\ast$ type make logics dynamical.


Every $\mathfrak{L}_X$ is an $X$-extended Boolean algebra. Denote its signature by
 $$
 L_X = \{ \vee, \wedge, \neg, \exists x, M_X \} ,\ \mbox{for all } x
\in X.
$$

 Here $M_X$ stands for the set of all symbols of relations of equality of the form $w\equiv w'$.

 Denote by  $S_{X,Y}$ the set of symbols of operations $s_*$ 
of the type $\tau=(X;Y)$, where $X,Y\in\Gamma$. Define the signature
 $$
 L_\Theta= \{L_X, S_{X,Y}; X, Y \in \Gamma \}.
$$
The signature $L_\Theta$ is multi-sorted. We  take $L_\Theta$  as the signature of an arbitrary algebra from the variety of multi-sorted algebras $Hal_\Theta$. The constructed multi-sorted algebras  $Hal_\Theta(H)$ possess this signature with the natural realization of all operations from $L_\Theta$.

 There is a bunch of axioms which determine algebras from the variety   $Hal_\Theta$. For example, every $s_\ast$ respects Boolean operations in $\mathfrak L_X$ and $\mathfrak L_Y$. Correlations of $s_\ast$ with equalities and quantifiers are described by more complex identities. Below we give the complete list of axioms for $Hal_\Theta$ (see also \cite{PlAlPl}, \cite{Plotkin_Gagta}).


\begin{defn}\label{def:ha}
We call an algebra $\frak L = (\frak L _X, X \in \Gamma)$ in the
signature $L_\Theta$ a Halmos algebra, if

\begin{enumerate}

\item
 Every domain $\frak L _X$ is an extended Boolean algebra in the signature $L_X$.
\medskip
\item
Every mapping $s_*: \frak L _X \to \frak L _Y$ is
  a homomorphism of Boolean
algebras. Let $s: W(X) \to W(Y)$, $s': W(Y) \to W(Z)$, and let
$u\in \frak L _X$. Then $s'_*(s_*(u))=(s's)_*(u)$.

\medskip

\item Conditions 
controlling the interaction of $s_{*}$ with
quantifiers are as follows:

\begin{enumerate}

\medskip
\item[(a)]
 $s_{1*} \exists x a = s_ {2*} \exists x
a, \ a \in \frak L (X)$, if $s_1( y) = s_2(y)$ for every $y\neq x$, $x$, $y\in X$.


 \medskip

 \item[(b)]
   $s_{*}\exists x a = \exists (s(x)) (s_*a),$ $\ a \in \frak L (X)$,
if $ s(x) = y $ and $y $ is a variable which does not belong to
the support
 of $s(x')$,
 for every
$x' \in X$ and $x' \neq x$.

\noindent
This condition means that $y$  does not participate in the shortest expression of the
element $s(x')\in W(Y)$. 

\end{enumerate}

\item
   Conditions 
   controlling the interaction of $s_{*}$ with equalities are as follows:
  \begin{enumerate}

  \medskip

  \item[(a)]
   $s_{*}(w\equiv w')=(s(w)\equiv s(w'))$.

   \medskip

  \item[(b)]
   $(s^{x}_{w})_{*}a \wedge (w\equiv w')\le (s^{x}_{w'})_{*}a$,
   where $a\in \frak L (X)$ and $s^{x}_{w}\in End(W(X))$ is defined by: $s^{x}_{w}( x)=w$
   and $s^{x}_{w}(x')=x',$ for $x'\ne x$.

 \end{enumerate}

\end{enumerate}

\end{defn}

\begin{remark} We should note that all conditions from the definition of a Halmos algebra can be represented as identities, and this is why the class of Halmos algebras is indeed a variety.
\end{remark}

 Define $Hal_\Theta$ to be the variety of all Halmos algebras, that is every algebra from $Hal_\Theta$ satisfies Definition \ref{def:ha}.

 \begin{prop}\label{pr:bool}
 Each algebra $Hal_\Theta(H)$ belongs the variety $Hal_\Theta$.
 \end{prop}

 This Proposition will be proved in Subsection \ref{sub:id}. Moreover,

   \begin{theorem}[\cite{Plotkin_AGinFOL}]\label{th:gen}
  All $Hal_\Theta(H)$, where $H$ runs $\Theta$, generate the variety $Hal_\Theta$.
\end{theorem}

In view of Theorem \ref{th:gen} one could define from the very beginning  the variety $Hal_\Theta$ as the variety, generated by all algebras $Hal_\Theta(H)$.

Recall, that every ideal of an extended Boolean algebra is a Boolean ideal invariant with respect to the universal  quantifiers action. Extended Boolean algebra is called simple if it does not have non-trivial ideals. In the multi-sorted case an ideal is a system of one-sorted ideals which respects all operations of  the form $s_\ast$. A multi-sorted Halmos algebra is simple if it does not have  non-trivial ideals. Algebras $Hal_\Theta(H)$ and their subalgebras are simple Halmos algebras, see \cite{BPlAlEPl}. Moreover, these algebras are the only simple algebras in the variety $Hal_\Theta$. Finally, every Halmos algebra is residually simple, see \cite{BPlAlEPl}. This fact is essential in the next subsection. Note, that all these facts are true because of the clever choice of the identities in the variety $Hal_\Theta$.

\subsection{Multi-sorted algebra of formulas}\label{sub:ms}

We shall define the algebra of formulas
$$
\widetilde \Phi = (\Phi(X),\ X\in \Gamma).
$$

 We define this algebra as the free over the multi-sorted set of equalities
  $$
 M=(M_X, X\in \Gamma)
 $$
 algebra  in $Hal_\Theta$. Assuming this property denote it as
 $$
Hal_\Theta^0=(Hal_\Theta^X, \ X\in \Gamma).
 $$
 So, $Hal_\Theta^X =\Phi(X)$ and $\widetilde \Phi=Hal_\Theta^0$.

 In order to define $Hal_\Theta^0$ we start from the  absolutely free over the same $M$ algebra
 $$\mathfrak L^0=(\mathfrak L^0(X), \ X\in \Gamma).
 $$
 This free algebra is considered in the signature of the variety $Hal_\Theta$. Algebra $
 \mathfrak L^0$ can be viewed as the algebra of pure formulas of the corresponding logical calculus.

 Then, $\widetilde \Phi$ is defined as the quotient algebra of $\mathfrak L^0$ modulo the verbal congruence of identities of the variety $Hal_\Theta$.
    The same algebra $\widetilde \Phi$ can be obtained from $\mathfrak L^0$ using the Lindenbaum-Tarski approach. Namely, basing on identities of $Hal_\Theta$ we distinguish in $ \mathfrak L^0$ a system of axioms and rules of inference. For every $X\in \Gamma$ consider the formulas
    $$
    (u\to v) \wedge (v\to u),
    $$
    where $u, v \in \mathfrak L^0(X)$. Here $u\to v $ means $ \neg u \vee v$. We assume that every $$
    (u\to v) \wedge (v\to u),
    $$
    is deducible from the axioms if and if the pair $(u,v)$ belongs to the $X$-component of the given verbal congruence.

So, $\widetilde \Phi$ can be viewed as an algebra of the compressed formulas modulo this congruence.


\subsection{Homomorphism $Val_H$}\label{sub:val}

Proceed from the mapping
$$
M_X \to Hal^X_\Theta(H),
$$
 which takes the equalities $w\equiv w'$ in $M_X$ to the corresponding equalities $[w\equiv w']_H$ in $Hal^X_\Theta(H)$.
  This gives rise also to the multi-sorted mapping
  $$
  M=(M_X, X\in \Gamma)\to Hal_\Theta(H)=(Hal^X_\Theta(H),\ X\in \Gamma).
  $$

  Since the multi-sorted set $M$ generates freely the algebra $\widetilde \Phi$ this
 mapping is  uniquely extended up to the homomorphism $$Val_H : \widetilde \Phi \to Hal_\Theta(H).$$
 Note that this homomorphism is a unique homomorphism from $\widetilde \Phi \to Hal_\Theta(H),$ since equalities are considered as constants.

  We have $$Val_H^X :  \Phi(X) \to Hal_\Theta^X(H),$$ i.e., $Val_H$ acts  componentwise for each $X \in \Gamma$.

  Recall that for every $u \in \Phi(X)$ the corresponding set $Val_H^X(u)$ is a set of points $\mu : W(X) \to H$ satisfying the formula $u$ (see Subsection \ref{sub:qq}). The logical kernel $LKer(\mu)$ was defined in Subsection \ref{sub:1} in these terms. Now we can say, that if a formula $u$ belongs to $\Phi(X)$ and a point  $\mu: W(X)\to H$ is given, then
  $$
  u\in LKer(\mu) \text{ if and only if } \mu \in Val^X_H(u).
  $$
We shall note that a formula $u$ can be, in general, of the form $u=s_*(v)$, where $v\in \Phi(Y)$, $Y$ is different from $X$. This means that the logical kernel of the point is very big and it gives a rich characterization of the whole theory.

Recall further that $LKer(\mu)$ is a Boolean ultrafilter containing the elementary theory $Th^X(H)$. Any ultrafilter with this property  will be  considered as an $X$-type of the algebra $H$.

  It is clear that
    $$
Ker(Val_H)=Th(H).
$$
This remark is used, for example, in
\begin{defn}%
An algebra $H \in \Theta$ is called saturated if for every $X\in \Gamma$  for each ultrafilter $T$ in $\Phi(X)$ containing $Th^X(h)$ there is a representation $T = LKer(\mu)$ for some $\mu : W(X) \to H$.
\end{defn}



Recall that the algebra $\widetilde \Phi$ is residually simple. This fact implies two important observations:

1. Let $u, v$ be two formulas in $\Phi(X)$. These formulas coincide if and only if for every algebra $H\in \Theta$ the equality
$$
Val^X_H(u)=Val^X_H(v)
$$
holds.

2. Let a morphism $s:W(X)\to W(Y)$ be given. It corresponds the morphism $s_\ast: \Phi(X)\to \Phi(Y)$. Let us take formulas $u\in \Phi(X)$ and $v\in \Phi(Y)$. The equality
$$
s_\ast(u)=v
$$
holds true if and only if for every algebra $H$ in $\Theta$ we have
$$
s_\ast(Val^X_H(u)=Val^Y_H(v).
$$


The following commutative diagram relates syntax with semantics
$$
\CD
\Phi(X) @> s_\ast >> \Phi(Y)\\
@V  \kp^X_H  VV @VV \kp^Y_H V\\
Hal_\Theta^X(H) @>s^H_\ast>> Hal_\Theta^Y(H)
\endCD
$$

We finished the survey of the notions of multi-sorted logic needed for UAG and in the next section we will relate these notions with the ideas of one-sorted logic used in Model Theory. Note also that we cannot define algebras of formulas $\Phi(X)$ individually. They are defined only in the multi-sorted case of algebras $\tilde \Phi = (\Phi(X), X \in \Gamma)$.

In fact, the definition of the algebra of formulas $\widetilde\Phi$ and the system of algebras $\Phi(X)$ is the main result of the first part of the paper. They are essentially used throughout the paper.

\subsection{Identities of the variety $Hal_\Theta$ for algebras $Hal_\Theta(H)$}\label{sub:id}

We have given already the definition of the algebras $Hal_{\Theta}(H)$. Now we show that
these algebras satisfy the axioms of Definition \ref{def:ha} and thus belong to the variety $Hal_\Theta$. In fact we should  check the correspondences
between $s_\ast$ and quantifiers and between $s_\ast$ and equalities.

First we consider interaction of $s_\ast$ with quantifiers. This
interaction is determined by  two following propositions.

\begin{prop}\label{Prop_Ch2_1}
Let $s_1$ and $s_{2}$ be morphisms $W(X)\to W(Y)$ and let $s_{1}(x')=s_{2}(x')$ for all
$x'\in X$, $x'\ne x$. Then the equality
$$
s_{1*}\exists x(A)=s_{2*}\exists x(A),
$$
where $A\subset Hom(W(X),H)$ holds in $Hal_{\Theta}(H)$.
\end{prop}

\proof 
 Let $\mu\in  s_{1*}\exists x(A)$ then $\mu s_{1}\in \exists
 x(A)$. In the set $A$ there is a point $\nu$ such that $\mu
 s_{1}(x')=\nu(x')$ for  $x'\ne x$, $x'\in X$. We also have the
 following equalities:
$$
\mu s_{2}(x')=\mu s_{1}(x')=\nu(x')
$$
and $\mu s_{2}\in \exists x(A)$. So $\mu\in s_{2*}\exists x (A)$. In a similar manner if
$\mu\in s_{2*}\exists x (A)$ then $\mu\in s_{1*}\exists x (A)$. Thus $s_{1*}\exists x
(A)=s_{2*}\exists x (A)$.

$\square$

Taking $A$ to be a point $a$ we obtain the axiom (3.a) of Definition \ref{def:ha}.
\medskip

\begin{prop}\label{Prop_Ch2_2}
Let $s:W(X)\to W(Y)$ be morphism. Take $x\in X$ and let $s(x)=y$ for some $y\in Y$. We
assume also that $y$ does not contain in the support of each $s(x')$, $x'\ne x$. Then
the equality
$$
s_{*}\exists x(A)=\exists s(x) s_{*}(A),
$$
where $A\subset Hom(W(X),H)$, holds in $Hal_{\Theta}(H)$.
\end{prop}

\proof
 Let $\mu\in \exists s(x) s_{*}(A)$. Take $\nu\in s_{*}A$ such
 that $\mu(y')=\nu(y')$, $y'\ne y=s(x)$, $y'\in Y$. We also have
 $\nu s=\gamma \in A$ and
 $$
  \mu(s(x'))=\mu s(x')=\nu(s(x'))=\nu s(x')=\gamma (x') 
 $$
 for every $x'\ne x$. So we have $\mu s\in \exists x(A)$ and $\mu \in s_{*}\big( \exists
 x(A)\big)$.

 Before  proving of the inverse inclusion we give some remarks.
 In first we generalize this situation. Instead of the one variable
 $x$ we will consider a set of variables $I$. Define the quantifier $\exists (I)$ by:
 $\mu\in \exists (I) A$ if there is a point $\nu$ in $A$
 such that $\mu(y)=\nu(y)$ for $y\not\in I$. Then we are
 interested in the following equality
 $$
  s_{*}\exists (I) A=\exists (s(I))s_{*}A.
 $$
 Let assume that $s(I)=J$ and $I\subset s^{-1}(J)$, and consider
 the equality $  s_{*}\exists (s^{-1}(J)) A=\exists (J)s_{*}A$. We
 will prove that it is true under the  condition: $s(x)=s(y)\in
 J$ if and only if $x=y$. Note that the latter condition follows from the
 assumption of our proposition.

 As before we check that if $\mu\in \exists (J) s_{*}A$
 then $\mu\in s_{*}\exists (s^{-1}(J)) A$.

 Let now $\mu\in s_{*}\exists (s^{-1}(J))
 A$. We will show that $\mu\in \exists (J) s_{*}A$. We have
 $\mu s\in \exists (s^{-1}(J)) A$ and $\nu\in A$ with $\mu s (y)=\nu
 (y)$ for all  $y\not\in s^{-1}(J)=I$.

 Now we choose the certain element $\gamma\in s_{*}A$. We assume
 that $\gamma(x)=\mu(x)$ for $x\not\in J$
 and $\gamma(x)=\nu(s^{-1}(x))$ if $x\in s(I)\subset J$.

 Take $x=s(x')$, $x'\in X$, $x\in J$. Then $x'=s^{-1}(x)$ and $x'$
 is uniquely defined by the element $x$. 
 So we have
$$
\gamma s(x')=\gamma (s(x'))=\nu(s^{-1}s(x'))=\nu(x'),
$$
where $x$ is an arbitrary element from the set $I$.

Let now $x'\not\in I$ and $s(x')=x$ does not belong to $J$. Then
$$
\gamma s(x')=\gamma(s(x'))=\mu (s(x'))=\mu s(x')=\nu(x').
$$
So we have $\gamma s(x')=\nu(x')$ for all $x'$ then  $\gamma s=\nu\in A$ and $\gamma\in
s_{*}A$. Thus $\mu\in \exists (J) s_{*}A$. As a result we have that
$$
s_{*}\exists (s^{-1}(J)) A=\exists (J)s_{*}A.
$$
We have started the proof of this equality  with the set $I$ and then turned to the set $s(I)=J$.
Using the condition $s(x)=s(y)$ implies $ x=y$ we have $s^{-1}(J)=I$. Now we can rewrite
the equality above as follows:
$$
s_{*}\exists (I) A=\exists (s(I))s_{*}A.
$$
If the set $I$ consist of only one element $x$ then we get that the statement of
Proposition~\ref{Prop_Ch2_2} holds.

$\square$

\medskip

Now we consider the correspondence between morphisms and equalities. Here we have two
conditions to check in $Hal_{\Theta}(H)$:
\begin{enumerate}
\item
 $s_{*}(w\equiv w')=\big( s(w)\equiv s(w')$\big),

\item
 $s^{x}_{w*}(A)\cap Val^{X}_{H}(w\equiv w')<
s^{x}_{w'*}A$,
\end{enumerate}
\noindent
 where $A\subset Hom(W(X),H)$.

We show that the first condition holds. Let $\mu: W(X)\to H$ be a point in
$s_{*}(w\equiv w')$. We have $\mu s \in Val^{X}_{H}(w\equiv w')$, $\mu s(w)=\mu s(w')$,
$(sw)^{\mu}=(sw')^{\mu}$, $\mu\in Val^{X}_{H}(s(w)\equiv s(w'))$.

Similarly we can check that if $\mu\in \big( s(w)\equiv s(w')\big)$ then  $\mu\in
s_{*}(w\equiv w')$.

Now we show that the second condition is true. Let
$$
\mu \in s^{x}_{w*}(A)\cap Val^{X}_{H}(w\equiv w').
$$
Then $\mu s^{x}_{w}\in A$ and $w^{\mu}=(w')^{\mu}$. From the last condition follows that
$\mu s^{x}_{w}(x)=\mu s^{x}_{w'}(x)$ and $\mu s^{x}_{w}(y)=\mu s^{x}_{w'}(y)$ for $y\ne x$.
This gives that $\mu s^{x}_{w}=\mu s^{x}_{w'}$. Since $\mu s^{x}_{w}\in A$ then $\mu
s^{x}_{w'}\in A$ and $\mu\in s^{x}_{w'*}(A)$.

Thus the correspondence between morphisms and equalities is verified.

So each algebra $Hal_{\Theta}(H)$ satisfies the identities of the variety $Hal_{\Theta}$.

\section{Logical geometry}

\subsection{Introduction}

The setting of logical geometry looks as follows. As before, we fix a variety of algebras $\Theta$.
Let $X=\{x_1, \ldots , x_n \}$ be a finite set of variables, $W(X)$ the free in $\Theta$ algebra over $X$,
$H$ an algebra in $\Theta$.  The set
$$Hom(W(X),H)$$
of all homomorphisms $\mu : W(X) \to H$ is viewed as the affine space of the sort $X$ over $H$.

Take the algebra of formulas $\Phi(X)$ which was defined in Subsection \ref{sub:ms}. Consider various subsets $T$ of $\Phi(X)$. We establish a Galois correspondence between such $T$ and sets of points $A$ in the space $Hom(W(X),H)$. This Galois correspondence gives rise to logical geometry in the given $\Theta$.

The notion  of the logical kernel  plays a major role in this correspondence. Recall (see Subsection \ref{sub:hom}),  that for every point $\mu:W(X)\to H$ in the algebra $\Phi(X)$ there exists it logical kernel $LKer(\mu)$, which is a Boolean ultrafilter in $\Phi(X)$, containing the elementary theory $Th^X(H)$.

Having in mind the context of the theory of models (see the next section), we view $LKer(\mu)$ as a LG-type (that is logically-geometric) type of the point $\mu$. Denote $LKer(\mu)=LG_H^X(\mu)$.

Note that the variety $\Theta$ is arbitrary and, correspondingly, the system of notions and statements of problems are of the universal character. However, even in the classical situation $\Theta=Com-P$ of the commutative and associative algebras with unit over the field $P$, a bunch of new problems and new results appear.

\subsection{Galois correspondence in the Logical Geometry}


Let us start with particular case when the set of formulas $T$ in $\Phi(X)$ is a set of equations of the form $w=w'$, $w, w' \in W(X)$, $X \in \Gamma$.


 We set
$$
A = T'_H = \{ \mu : W(X) \to H \ | \  T \subset Ker(\mu)\}.
$$
Here $A$ is an algebraic set in $Hom(W(X),H)$, determined by the set $T$.

Let, further, $A$ be a subset in $Hom(W(X),H)$. We set
$$
T = A'_H = \bigcap_{\mu \in A} Ker(\mu).
$$
Congruences $T$ of such kind are called  $H$-closed in $W(X)$. We have also Galois-closures $T''_H$ and $A''_H$.

Let us pass to general case of logical geometry. Let now $T$ be a set of arbitrary formulas in $\Phi(X)$. We set
$$
A = T^L_H = \{ \mu : W(X) \to H \ | \  T \subset LKer(\mu)\}.
$$
We have also
$$
A = \bigcap_{u \in T} Val^X_H(u).
$$
Here $A$ is called a definable  set in $Hom(W(X),H)$, determined by the set $T$ (cf., Section \ref{sub:mt}). We use the term "{\it definable}" for $A$ of such kind, meaning that $A$ is defined by some set of formulas $T$.

 For the set of points $A$  in $Hom(W(X),H)$ we set
$$
T = A^L_H = \bigcap_{\mu \in A} LKer(\mu).
$$


We have also 
$$
T = A^L_H = \{ u\in \Phi(X) \ |\  A \subset Val^X_H(u)\}.
$$

 Here $T$ is a Boolean filter in $\Phi(X)$ determined by the set of points $A$. Filters of such kind are Galois-closed and we can define the Galois-closures of arbitrary sets $T$ in $\Phi(X)$ and $A$ in $Hom(W(X),H)$ as $T^{LL}$ and $A^{LL}$.

\begin{prop}\cite{BPlAlEPl}\label{Prop_ClosedFiltAlg}
Intersection of $H$-closed filters is also $H$-closed filter.
\end{prop}

 \subsection{AG-equivalent and LG-equivalent algebras. LG-isotypic algebras}
Let us formulate two key definitions and two results (see, for example,  \cite{Pl-St}, \cite{Plotkin_SomeResultsUAG}).

\begin{defn}\label{def:AG} {
Algebras $H_1$ and $H_2$ are AG-equivalent, if $T''_{H_1}=T''_{H_2}$ always holds true.}
\end{defn}

\begin{defn}\label{def:lgeq} {
Algebras $H_1$ and $H_2$ are LG-equivalent, if $T^{LL}_{H_1}=T^{LL}_{H_2}$ always holds true.}
\end{defn}

Let now
$$
\bigwedge_{(w,w')\in T} w\equiv w' \to w_{0}\equiv w'_{0}
$$
be a quasiidentity. We will also write
$$
T\to w_{0}\equiv w'_{0}.
$$
This quasiidentity can be infinitary if the set $T$ is infinite. Note that $w_{0}\equiv
w'_{0}\in T''_{H}$ if and only if the quasiidentity $ T\to  w_{0}\equiv w'_{0}$ holds true
in the algebra $H$.

Now algebras $H_1$ and $H_2$ in $\Theta$ are $AG$-equivalent, if and only if each
quasiidentity $T\to w_{0}\equiv w'_{0}$ which
holds true in $H_{1}$ is a quasiidentity of the algebra $H_2$.

In particular, if $H_1$ and $H_2$ are $AG$-equivalent then they generate the same
quasi-variety. The inverse statement is not true (see \cite{MR}). Recall that
quasi-varieties are generated by systems of finitary quasi-identities.

Consider the following formula:
$$
\bigwedge_{u\in T} u \to v, \ v\in \Phi(X)
$$
or
$$
T\to v.
$$

The set $T$ can be infinite and then we speak about infinitary formula.

\begin{prop}
A formula $v$ belongs to $T^{LL}_{H}$ if and only if the formula $T\to v$ holds true in the
algebra $H$.
\end{prop}
\begin{proof}
Take $A=T^{L}_{H}$. We have $v\in T^{LL}_{H}$ if and only if $A\subset Val^{X}_{H}(v)$. A
point $\mu$ belongs to $A$ if and only if $\mu$ satisfies every $u\in T$. We have that the
formula $T\to v$ holds true in $H$ if and only if for every point $\mu$ satisfying all
formulas $u\in T$ this point satisfies the formula $v$, i.e. $\mu\in Val^{X}_{H}(v)$. So
$A\subset Val^{X}_{H}(v)$ at the same time as $T\to v$ holds in $H$.
\end{proof}

From this proposition follows:
\begin{prop}\label{pr:imp}
Algebras $H_1$ and $H_2$ are $LG$-equivalent if for every $X\in \Gamma$ and $T\subset
\Phi(X)$ the formula $T\to v$ holds true in the algebra $H_1$ if and only if it is true in
the algebra $H_2$.
\end{prop}

Denote by $ImTh(H)$ the implicative theory of the algebra $H$. Remind that the implicative
theory is the set of all formulas of the form $T\to u$, for different $X\in \Gamma$, which
hold true in the algebra $H$. So algebras $H_1$ and $H_2$ are $LG$-equivalent if their
implicative theories are coincided i.e.,
$$
ImTh(H_1)=ImTh(H_2).
$$

\bigskip

Now we give one more approach to the $LG$-equivalence.
Let $T$ be a set of formulas from $\Phi(X)$ and let $T^{\vee}$ be the set of all
disjunctions of the formulas $u\in T$ and $\widetilde T^{\vee}$ be the set of all
disjunctions of the formulas $\neg u$ for $u\in T$. Here we have the following properties
$$
\neg(\bigwedge_{u\in T}u)=\widetilde T^{\vee};\ \neg(\bigwedge_{u\in T}\neg u)=T^{\vee}.
$$

We want to consider the disjunctive theory of the algebra $H$. The disjunctive theory of
the algebra $H$ is the set of all possible formulas $T^{\vee}$, for all $T\subset \Phi(X)$
and different $X\in \Gamma$, which hold true in the algebra $H$.

Note that the formula $T\to v$ holds true in the algebra $H$ if and only if the formula
$\widetilde T^{\vee}\vee v$ is true in $H$. Thus if the disjunctive theories of two
algebras $H_1$ and $H_2$  coincide then these algebras are $LG$-equivalent. Moreover
there is the following
\begin{prop}\label{pr:dis}
Algebras $H_1$ and $H_2$ are $LG$-equivalent if and only if their disjunctive theories
coincide.
\end{prop}

\begin{proof}
 Let algebras $H_1$ and $H_2$ be $LG$-equivalent.
 We take a set of formulas $T\subset
 \Phi(X)$ and consider the formula $\widetilde T^{\vee} \vee v$,
 where $v$ is the following formula $(x\equiv y)\wedge (x\not\equiv
 y)$. There is no point $\mu:W(X)\to H$ satisfying the formula
 $v$. So $\mu$ satisfies the formula $\widetilde T^{\vee} \vee v$
 if and only if $\mu$ satisfies the formula $\widetilde T^{\vee}$.
 It means that there is  $u\in T$ such that the point $\mu$ does
 not satisfy the formula $u$ and so this point satisfies the
 formula $T\to v$.

 Now let $H=H_1$ and let $T^{\vee}$ be the formula which  hold true in the algebra
 $H_1$. An arbitrary point $\mu_{1}:W(X)\to H_1$ satisfies
 $T^{\vee}$ and $\widetilde T^{\vee} \to
 v$. Since the algebras $H_1$ and $H_2$ are $LG$-equivalent then
 every point $\mu_{2}:W(X)\to H_2$ satisfies the formula $\widetilde T^{\vee} \to
 v$ and, hence, it satisfies the formula $T^{\vee}\vee v$.
 So  each point $\mu_{2}:W(X)\to H_2$ satisfies the formula
 $T^{\vee}$. Thus the disjunctive theories of $H_1$ and $H_2$
  coincide.
\end{proof}

Note that 
\begin{prop}
 If algebras $H_1$ and $H_2$ are
$LG$-equivalent then they are elementarily equivalent.
\end{prop}

\begin{proof}
 Indeed, let consider the formula $u\to v$, where $u$ is the
 formula $x\equiv x$. This formula holds true in the algebra $H$
 if and only if the formula $v$ is true in $H$, i.e. $v\in Th(H)$.
 If algebras $H_1$ and $H_2$ are $LG$-equivalent then the formula $u\to v$
 holds in $H_1$ if and only if it is true in $H_2$.
 Thus $v\in Th(H_1)$ if and only if $v\in Th(H_2)$, that is $Th(H_1)=Th(H_2)$.
\end{proof}

\begin{defn}\label{def:iso}
 Two algebras $H_1$ and $H_2$ are called LG-isotypic (cf. Section \ref{sec:isotyp}) if for every point $\mu:W(X)\to H_1$ there exists a point $\nu:W(X)\to H_2$ such that $LKer(\mu)=LKer(\nu)$ and, conversely, for every point $\nu:W(X)\to H_2$ there exists a point $\mu:W(X)\to H_1$ such that $LKer(\nu)=LKer(\mu)$.
\end{defn}

The main theorem is the following
\begin{theorem}\label{thm:lgiso} 
Algebras $H_1$ and $H_2$ are $LG$-equivalent if and only if they are LG-isotypic.
\end{theorem}

\begin{proof}
 %
  Let $H_1$ and $H_2$ be $LG$-equivalent algebras.
 By  definition 
 for any finite set $X$,
 and any $H_1$-closed filter $T$ from  $\Phi(X)$ we have:
$$
T=T^{LL}_{H_1}=T^{LL}_{H_2}.
$$
So, $T$ is $H_1$-closed if and only if it is $H_2$-closed.

Let $T=LKer(\mu)$ be the logical kernel of a point $\mu:W(X) \to H_1$. Then $T^L_{H_1}=A$, where $A=\{\mu\}$ and $T^{LL}_{H_1}=A^L_{H_1}=LKer(\mu)=T$. So, $T$ is an $H_1$-closed filter. Hence, $T$ is an $H_2$-closed filter. Since $T=LKer(\mu)$, the filter $T$ is maximal. Since $H_1$ and $H_2$ are $LG$-equivalent, there exists a set $B$ in $Hom(W(X),H_2)$ such that $B^L_{H_{2}}=T$. Then $T=\bigcap_{\nu\in B}LKer(\nu)$. Since the filter $T$ is maximal, $LKer(\nu)=T=LKer(\mu)$ for all points $\nu\in B$.

Note that we used the fact that $T^{L}_{H_2}$
is not empty. Indeed, if we assume that $T^L_{H_2}=\{ \emptyset \}$ then $T^{LL}_{H_2}=\{
\emptyset \}^{L}_{H_2}=\Phi(X)=T^{LL}_{H_1}=T$, but
$T$ is a proper filter.

%


In the similar way one can prove that if $T=LKer(\mu)$ is the logical kernel of a point $\nu:W(X) \to H_2$, then there exists a point $\mu:W(X) \to H_1$ such that $ LKer(\nu) =  LKer(\mu).$ Hence, $H_1$ and $H_2$ are isotypic.

Let, further, $H_1$ and $H_2$ be isotypic algebras. This means that if
$T=LKer(\mu)$ is the logical kernel of a point $\mu:W(X) \to H_1$, then $T=LKer(\nu)$ is
the logical kernel for some $\nu:W(X) \to H_2$ as well, and vice versa. Recall, that every
logical kernel is a closed filter, so $T$ is $H_1$- and $H_2$-closed filter.

Let, now, $T$ be an arbitrary $H_1$-closed filter in $\Phi(X)$. We will show that $T$ is
$H_2$-closed.

Let $T^L_{H_1}=A$, then $T=T^{LL}_{H_1}=A^{L}_{H_1}= \bigcap_{\mu \in A}LKer(\mu)$. Since $H_1$ and $H_2$ are isotypic, there exist points $\nu:W(X) \to H_2$ such that $\bigcap_{\nu \in Hom(W(X),H)}LKer(\nu)=\bigcap_{\mu \in A}LKer(\mu).$
%
According to Proposition~\ref{Prop_ClosedFiltAlg}, the intersection of $H$-closed filters
is also $H$-closed filter, hence $T$ is an $H_2$-closed filter.

Similarly, we can prove that each $H_2$-closed filter  is $H_1$-closed.  Hence, $H_1$ and
$H_2$ are $LG$-equivalent.
\end{proof}

From this theorem follows
\begin{corol}\label{cor:el}
If the algebras $H_1$ and $H_2$ are isotypic, then they are elementarily equivalent.
\end{corol}
\begin{proof}
Take a formula $x=x\to u$, where $u\in \Phi(X)$.  This formula holds in $H_1$ if and only if $u$ holds in $H_1$. Since $H_1$ and $H_2$ are isotypic, then (Proposition \ref{pr:imp})  $x=x\to u$ holds in $H_1$ if and only if it holds in $H_2$. So if $u$ belongs to the elementary theory of $H_1$, then it belongs to the elementary theory of $H_2$ and vice versa.

\end{proof}

\subsection{Categories of algebraic and definable sets for a given algebra $H$}


Recall that we introduced (Section \ref{sub:hal})  the  category of affine spaces
$\Theta^*(H)$. It is natural to assume that $Var(H)=\Theta$. If this condition does not hold, the situation when for two different morphisms $s_1: W(Y)\to W(X)$ and $s_2: W(Y)\to W(X)$ the corresponding morphisms $\widetilde s_1$ and $\widetilde s_2$ in $\Theta^*(H)$ coincide , is possible. This breaks duality between $\Theta^0$ and $\Theta^*$ (Proposition \ref{prop:dual}) and, as we will see, leads to a lot of other disadvantages. The condition $Var(H)=\Theta$
 plays also a crucial role in the problem of sameness of geometries over different algebras.



Define now a category of algebraic sets $AG_\Theta(H)$ and a
category of definable sets $LG_\Theta(H)$.

Define first a category  $Set_\Theta(H)$. Its objects are pairs $(X,A)$ with $A$ a subset in $Hom(W(X),H)$ and $X \in \Gamma$.

A morphism $s_*$ takes $(X, A)$ to $(Y, B)$, where $s: W(Y)\to
W(X)$ and $B$  contains the points $\nu: W(Y)\to H$ such that
$\nu=\mu s$, for $\mu\in A$.

Now, $AG_\Theta(H)$ is a full subcategory in $Set_\Theta(H)$,
whose objects are pairs $(X,A)$, where $A$ is an algebraic set.

If for $A$ we take definable sets, then we have the category
$LG_\Theta(H)$ which is a full subcategory in $Set_\Theta(H)$.


Two key results are as follows (see, for example,  \cite{Pl-St}, \cite{Plotkin_SomeResultsUAG}).



\begin{theorem}\label{th:agc} {
If $H_1$ and $H_2$ are AG-equivalent, then the categories $AG_\Theta(H_1)$ and  $AG_\Theta(H_2)$ are isomorphic.}
\end{theorem}

\begin{theorem}\label{th:agl} {
If $H_1$ and $H_2$ are $LG$-equivalent, then the categories
$LG_\Theta(H_1)$ and  $LG_\Theta(H_2)$ are isomorphic.}
\end{theorem}

\begin{remark}
In view of Theorem \ref{def:iso}, the geometric notion of
$LG$-equivalence of algebras is equivalent to a model theoretic
notion of isotypic algebras. Thus, if algebras $H_1$ and $H_2$ are
isotypic, then the categories of definable sets $LG_\Theta(H_1)$
and  $LG_\Theta(H_2)$ are isomorphic for every $\Theta$.
\end{remark}

Theorems \ref{th:agc} and \ref{th:agl} provide sufficient conditions for isomorphisms of categories of algebraic and definable sets, respectively. Necessary and
other sufficient conditions
will be considered in \ref{}. 

Beforehand, we shall slightly modify the categories $AG_\Theta(H)$ and $LG_\Theta(H)$.  First of all, modify the definition of the category $AG_\Theta(H)$.
Objects $AG_\Theta^X(H)$ of $AG_\Theta(H)$ are not pairs $(X,A)$, where $A$ is an algebraic set, but systems of all algebraic sets in the space $Hom(W(X),H)$, where $X$ is fixed. Analogously, an object  $LG_\Theta^X(H)$ is the system of all definable sets in the space $Hom(W(X),H)$.

Note that all definable sets under the given $X$ constitute a lattice, while all algebraic sets are just a poset.  So, one can say that objects $AG_\Theta^X(H)$ of $AG_\Theta(H)$ are posets of algebraic sets in $Hom(W(X),H)$, objects $LG_\Theta^X(H)$ of $LG_\Theta(H)$ are lattices of definable sets in the space $Hom(W(X),H)$. By definition, every algebraic set is a definable set.

Morphisms between $AG_\Theta^X(H)$ and $AG_\Theta^Y(H)$, as well as between $LG_\Theta^X(H)$ and $LG_\Theta^Y(H)$  are defined towards the maps $s: W(Y)\to W(X)$. We will describe these morphisms in more detail.


First of all, recall that objects in the categories $\Theta^0$ and $\widetilde\Phi_\Theta$ are free algebras $W(X)$ and algebras of formulas $\Phi(X)$, respectively. Every homomorphism $s:W(Y)\to W(X)$ gives rise to a morphism $s_\ast: \Phi(Y)\to \Phi(X)$. In particular, $s_\ast$ acts on equalities as follows: $s_\ast(w_1\equiv w_2)=(s(w_1)\equiv s(w_2))$ (action of $s_\ast$ is regulated by Definition \ref{def:ha}). Note that equalities of the form $w\equiv w'$, $w$, $w'$ in $W(X)$,  can be treated as formulas in $\Phi(X)$.  This correspondence $s\to s_\ast$ allows us to define morphisms $\widetilde s$ and $\widetilde s_\ast$ in $AG_\Theta(H)$ and $LG_\Theta(H)$.

Given  $s:W(Y)\to W(X)$, a morphism $\widetilde s: AG_\Theta^X(H)\to AG_\Theta^Y(H)$ is defined as follows: 
for an algebraic set $A$ in $AG_\Theta^X(H)$  take all points $\nu$ in $Hom(W(Y),H)$ of the form $\nu=\mu s$, where $\mu\in A$ and define $B=\widetilde s A$ as the algebraic set determined by the set of all such $\nu$. Then the object  $AG_\Theta^Y(H)$ corresponding to $AG_\Theta^X(H)$ contains  all $B$ of such kind.
So, morphisms in $AG_\Theta(H)$ are maps of posets, originated from homomorphisms of free algebras, that is maps of the form $\widetilde s$. Note, that all $\widetilde s$ preserve posets structure by definition.

Analogously, a morphism $\widetilde s_\ast: LG_\Theta^X(H)  \to
LG_\Theta^Y(H)$ is defined as follows: given $A\in LG^X_\Theta(H)$ and $s:W(Y)\to W(X)$, the  set $B=\widetilde s_\ast A$ is the definable set determined by all points $\nu$ of the form $\nu=\mu s$, $\mu \in A$. The object  $LG_\Theta^Y(H)$ corresponding to $LG_\Theta^X(H)$ contains  all $B$ of such form.

Now we define categories of algebras of formulas $C_\Theta(H)$ and $F_\Theta(H)$. Let us start with $C_\Theta(H)$.  If $A\in AG^X_\Theta(H)$, then take $T=A'_H$. This is an $H$-closed congruence on $W(X)$, that is $T_H'=A$. Denote by $C_\Theta^X(H)$ the poset of all such $T$, where $A$ runs $AG^X_\Theta(H)$. These $C_\Theta^X(H)$ are objects of $C_\Theta(H)$. They are in one-to-one correspondence with objects $AG_\Theta^X(H)$. 

 Now about morphisms. Let $s: W(Y)\to W(X)$ be a morphism in $\Theta^0$. Recall that $s_\ast(w_1\equiv w_2)=(s(w_1)\equiv s(w_2))$. Let $T_2$ be an $H$-closed congruence in $C_\Theta^Y(H)$. Define $T_1$ as the $H$-closed congruence in $C_\Theta^X(H)$ determined by the set of all equalities of the form $s_\ast(w\equiv w')$, where
 $w\equiv w'$ in $T_2$. So $T_1=(s_\ast T_2)''$.

 Consider the commutative diagram $(\lozenge)$
$$
\CD
T_2 @> s_\ast >> T_1\\
@V \al_{H}^Y  VV @VV \al_{H}^X V\\
B @< \widetilde{s} << A,
\endCD
$$
where $A'_H=T_1$, $T'_{1H}=A$, $B'_H=T_2$, $T'_{2H}=B$ (follows from Section \ref{sub:val}). Here $T_2$ and $T_1$ are $H$-closed congruences in $W(Y)$ and $W(X)$, respectively. In particular $(\lozenge)$ implies that $s_\ast: C_\Theta^Y(H)\to C_\Theta^X(H)$ is the map of posets.

 This diagram gives rise to the category $C_\Theta(H)$ of all $H$-closed congruences.

 It is important to get another look at the morphisms $s_\ast$ in $C_\Theta(H)$. Let the $H$-closed congruences $T_2$ in $C_\Theta^Y(H)$ and $T_1$ in $C_\Theta^X(H)$ be given. A morphism $s_\ast$
 takes $T_2$ to $T_1$, if and only if $s_\ast$ satisfies the diagram $(\lozenge)$. So, $s_\ast$ assigns $T_1$ to $T_2$ if and only if we have  $(\lozenge)$. Moreover, if one knows $s_\ast$ and $T_1$, then $(\lozenge)$ recovers $T_2$. Thus, one can define the category $C_\Theta(H)^{-1}$, with the same objects as $C_\Theta(H)$ and morphisms acting opposite-wise.





\begin{proposition}\label{pr:duali}
Let $Var(H)=\Theta$. The category $C_\Theta(H)$ of posets of $H$-closed congruences is anti-isomorphic to the category $AG_\Theta(H)$ of posets  of  algebraic sets.
\end{proposition}
\begin{proof}
The correspondence $C_\Theta(H)\to AG_\Theta(H)$ is one-to-one. The condition $Var(H)=\Theta$ provides that the correspondence $s_\ast\to \widetilde s$ is also one-to-one (see Proposition \ref{prop:dual}).
\end{proof}
 The opposite category $C^{-1}(H)$ with morphisms $s_\ast^{-1}: T_1\to T_2$ is isomorphic to $AG_\Theta(H)$.

We shall repeat the similar construction using $L$-Galois correspondence. We have the  diagram $(\diamondsuit\diamondsuit)$ (whose particular case is the  diagram $(\diamondsuit)$ :

$$
\CD
T_2 @> s_\ast >> T_1\\
@V \al_{H}^Y  VV @VV \al_{H}^X V\\
B @< \widetilde{s}_\ast << A,
\endCD
$$
where $A^L_H=T_1$, $T^L_{1H}=A$, $B^L_H=T_2$, $T^L_{2H}=B$. Here $T_2$ and $T_1$ are $H$-closed filters in $\Phi(X)$ and $\Phi(Y)$, respectively. It gives rise to the categories of $H$-closed filters $F_\Theta(H)$ and $F^{-1}_\Theta(H)$. Objects of $F_\Theta(H)$ are lattices of $H$-closed filters $F_\Theta^X(H)$.
Let $F_2$ be an $H$-closed filter in $F_\Theta^Y(H)$. Define $F_1$ as the $H$-closed filter determined by the set of formulas of the form $s_\ast v$, where
 $v$ in $T_2$. So, $F_1=(s_\ast F_2)^{LL}$.

 In other words, let the $H$-closed filters $T_2$ and $T_1$ in $F_\Theta^Y(H)$ and  $F_\Theta^X(H)$, respectively, be given. Take $T_{1H}^L=A$ and $T_{2H}^L=B$. The diagram  $(\diamondsuit\diamondsuit)$  determines when $s_\ast$ takes $T_2$ to $T_1$. In particular, $T_1$ defines uniquely $T_1$ by $T_2=s_\ast^{-1}(T_1)$, that is $T_2$ is the inverse image of $T_1$.

 \begin{proposition}\label{pr:duali1}
Let $Var(H)=\Theta$. The category $F_\Theta(H)$ of lattices of $H$-closed filters is anti-isomorphic to the category $LG_\Theta(H)$ of lattices  of  definable sets.
\end{proposition}

 The category $F^{-1}_\Theta(H)$ is isomorphic to the category of definable sets $LG_\Theta(H)$.

\subsection{Geometric and logical similarity of algebras}\label{sec:lgsim}

\begin{defn}\label{def:gsim} We call algebras $H_1$ and $H_2$ are called geometrically similar if the categories of algebraic sets $AG_\Theta(H_1)$ and $AG_\Theta(H_2)$ are isomorphic.
\end{defn}

Since the categories $AG_\Theta(H)$ and $C_\Theta(H)$ are dual, algebras $H_1$ and $H_2$ are  geometrically similar if and only if the categories $C_\Theta(H_1)$ and $C_\Theta(H_2)$ are isomorphic. In view of Theorem \ref{def:AG}, if algebras $H_1$ and $H_2$ are  geometrically equivalent, then they are geometrically similar.

\begin{defn}\label{def:lsim} We call algebras $H_1$ and $H_2$ are called logically similar, if the categories of definable sets $LG_\Theta(H_1)$ and $LG_\Theta(H_2)$ are isomorphic.
\end{defn}

Algebras $H_1$ and $H_2$ are  logically similar if and only if $F_\Theta(H_1)$ and $F_\Theta(H_2)$ are isomorphic. 


 By Theorem \ref{def:lgeq}  if  $H_1$ and $H_2$ are  logically equivalent, then they are logically similar.


The following problems is our main target:

\begin{problem}\label{pr:2_1}
Find necessary and sufficient conditions on algebras $H_1$ and
$H_2$ in $\Theta$ that provide algebraic similarity of these algebras.
\end{problem}

\begin{problem}\label{pr:3_1}
Find necessary and sufficient conditions on algebras $H_1$ and
$H_2$ in $\Theta$ that provide logical similarity of these algebras.
\end{problem}


 We start with examples of  specific varieties, where necessary and sufficient conditions for isomorphism of the categories of algebraic sets can be formulated solely in terms of properties of algebras $H_1$ and $H_2$. Afterwards we will dwell on a some general approach.
 In what follows, all fields and rings  are assumed to be infinite. 

\begin{theorem}{}\label{th:gsim}
Let $Var(H_1)=Var(H_2)=\Theta.$

\begin{enumerate}

\item Let $\Theta$ be one of the following varieties
\begin{itemize}
\item
$\Theta=Grp$, the variety of groups,
 \item
 $\Theta=Jord$, the variety of Jordan algebras,
 \item  $\Theta=Inv$, the variety of inverse semigroups,
 \item  $\Theta=\frak N_d$, the variety of nilpotent groups of class $d$.
 \end{itemize}
 Categories $AG_\Theta(H_1)$ and $AG_\Theta(H_2)$ are
isomorphic if and only if the algebras $H_1$ and $H_2$ are
geometrically equivalent (see \cite{Formanek}) \cite{MSZ}, \cite{Ts4},
\cite{Ts2}).

\item Let $\Theta=Com-P$ or $Lie-P$
and $\sigma \in Aut (P)$.  Define a new algebra $H^\sigma$.   In
$H^\sigma$ the multiplication on a scalar $\circ$ is defined
through the multiplication in $H$ by the rule:
\[
\lambda\circ a=\lambda^\sigma 
\cdot a, \;\; \;   \lambda \in P, \;\; \;  a \in H.
\]
\noindent
Categories $AG_\Theta(H_1)$ and $AG_\Theta(H_2)$ are
isomorphic if and only for some $\sigma\in Aut(P)$ the algebras
$H^{\sigma}_{1}$ and $H_2$ are geometrically equivalent (see \cite{Berzins_GeomEquiv},\cite{Pl-St}, \cite{MPP1},\cite{MPP2}, \cite{L}, \cite{S}).

\item Let $\Theta=Ass-P$. Denote by $H^{*}$ the algebra with the multiplication $*$
defined as follows: $a*b=b\cdot a$. The algebra $H^{*}$ is called
opposite to $H$. The categories $AG_{\Theta}(H_1)$ and $AG_{\Theta}(H_2)$ are
 isomorphic if and only if for some $\sigma\in Aut(P)$
the algebras
 $(H_1^*)^\sigma$ and $H_2$ are geometrically equivalent, where $(H_1^*)^\sigma$ is opposite to either  $H_1$ or to $H_1^*$ ( \cite{BBL}, \cite{Pl-St}).

\end{enumerate}

\end{theorem}

\begin{remark} The list of varieties of Theorem \ref{th:gsim} is not complete. The similar results are known for the varieties of semigroups \cite{MS},  linear algebras \cite{Ts3},\cite{S}, power associative algebras, alternative algebras, \cite{Ts4}, non-commutative non-associative algebras, commutative non-associative algebras, color Lie superalgebras, Lie p-algebras , color Lie p-superalgebras, Poisson algebras \cite{S}, of free $R$-modules \cite{KLP}, Nielsen-Schreier varieties \cite{S},  and for the varieties of some classes of representations \cite{PZ2}, \cite{TP}, \cite{Ts5}.
\end{remark}

We will  make some preparations, basing on the idea of isomorphism of functors.  


\begin{defn}\label{def:isom}
Let $\varphi_1, \varphi_2$ be two functors from a category $C_1$
to $C_2$. We say that an isomorphism of functors $S:\varphi_1 \to \varphi_2$  is defined if for any  morphism   $\nu: A\to B$ in $C_1$  the following commutative
diagram takes place

$$
\CD
\varphi_1(A) @> S_{A} >> \varphi_2(A)\\
@V \varphi_1(\nu)  VV @VV \varphi_2(\nu) V\\
\varphi_1(B) @>S_{B} >> \varphi_2(B).
\endCD
$$
Here $S_A$ is the  $A$-component of $S$, that is, a function which makes a bijective correspondence
 between $\varphi_1(A)$ and $\varphi_2(A)$.  The same is valid for $S_B$.
\end{defn}

Note that $S_A$ and $S_B$ are not necessarily morphisms in $C_2$. Thus, this definition is different from the standard one, where all $S_A$ have to be morphisms in $C_2$.  The commutative diagram above can be reformulated as
$$
\varphi_1(\nu)=S^{-1}_{B}\varphi_2(\nu)S_{A}, \quad \varphi_2(\nu)=S_{B}\varphi_1(\nu)S^{-1}_{A}.
$$




An invertible functor from a category to itself is an automorphism of a category. The notion of isomorphism of functors gives rise to the notion of the inner automorphism of a category.  An automorphism $\varphi$ of the category $C$  is called inner (see \cite{Pl-St}) if $\varphi $ is isomorphic to the identity functor $1_C$. This provides the commutative diagram
$$
\CD
A @> s_{A} >> \varphi(A)\\
@V \nu  VV @VV \varphi(\nu) V\\
B @>s_{B} >> \varphi(B),
\endCD
$$
that is $\varphi(\nu)=s_{B}\nu s^{-1}_{A}.
$


The following Proposition plays an important role in the proof of Theorem \ref{th:gsim}:

\begin{proposition}[\cite{Pl-Sib}]\label{prop:eq} If for the variety $\Theta$ every automorphism of the category $\Theta^0$ is inner,
then  two algebras $H_1$ and $H_2$ are geometrically similar if and only if they are geometrically equivalent.
\end{proposition}

So, studying automorphisms of $\Theta^0$ play a crucial role in Problem \ref{pr:2_1} related to geometric similarity of algebras.  The latter problem is reduced by Reduction Theorem  to studying the group $Aut(End(W(X))$ of automorphisms of the endomorphism semigroup of a free algebra (see \cite{Pl-St}, \cite{KLP}, \cite{MPP2}).



Now we will treat the general problem using the Galois-closure functors.

For every algebra $H\in\Theta$ consider two functors
$$Cl^A_H:\Theta^0\to PoSet,$$
$$Cl^L_H:\widetilde \Phi_\Theta\to Lat,$$
where $A$ and $L$ stand for the functors of algebraic  and logical closures, respectively. We will suppress these indexes in the sequel, assuming that the type of $Cl$-functor is clear in each particular case.

In fact, $PoSet$ is the category $C_\Theta(H)$ of partially ordered sets of $H$-closed congruences $C^X_\Theta(H)$, while $Lat$ is the category $F_\Theta(H)$ of lattices of $H$-closed filters $F_\Theta^X(H)$.

So, $Cl_H$ assigns to every object $W(X)$ in $\Theta^0$ the poset $C_\Theta^X(H)$ of all $H$-closed congruences on $W(X)$. If $s:W(Y)\to W(X)$ is a morphism in $\Theta^0$,  then $Cl_H(s)=s_\ast: C_\Theta^Y(H)\to C_\Theta^X(H)\to$ is a morphism in $C_\Theta(H)$.

Analogously, in case of $\widetilde \Phi_\Theta\to Lat$, every $s:W(Y)\to W(X)$ gives rise to
$$
s_\ast:\Phi(Y)\to \Phi(X),
$$
and for $T_2\subset \Phi(Y)$, $T_1\subset \Phi(X)$ define $
s_\ast: T_2\to T_1
$
by taking all $v\in T_2$ such that $s_\ast v=u\in T_1$. Using  $(\diamondsuit\diamondsuit)$ we extend $s_\ast$ to
$$
s_\ast: Cl_H(T_2)\to Cl_H(T_1)
$$
The correspondence $s\to s_\ast$ gives rise to the contravariant $Cl_H$  functor $\widetilde \Phi\to F_\Theta(H)$.






Apply these notions to  Problem \ref{pr:2_1} and Problem \ref{pr:3_1}.
Consider two commutative diagrams:

$$
\CD
\Theta^{0} @[2]> \varphi  >> \Theta^{0}\\
 @[2]/SE/ Cl_{H_1} //@.@.\;    @/SW// Cl_{H_2} /\\
 @. PoSet
 \endCD
 \qquad\qquad
 \CD \tilde\Phi @[2]>  \varphi >> \tilde\Phi \\
@[2]/SE/ Cl_{H_1} // @.@. \; @/SW//  Cl_{H_2} / \\
@. Lat
 \endCD
$$
\smallskip
where $\varphi$ is an automorphism of $\Theta^0$ or $\widetilde \Phi$. Commutativity of these diagrams means that there exists an
isomorphism of  functors
$$
s=s(\varphi):Cl_{H_1} \to Cl_{H_2} \cdot \varphi.
$$


In its turn, this isomorphism of functors means that the diagram

 $$
\CD
Cl_{H_1}(W(Y)) @> s_{W(Y)} >> Cl_{H_2}(\varphi(W(Y))) \\
 @V Cl_{H_1}(\nu) VV @VV Cl_{H_2}\varphi(\nu) V\\
 Cl_{H_1}(W(X)) @> s_{W(X)} >> Cl_{H_2}(\varphi(W(X))),
\endCD
$$

and the diagram

$$
\CD
Cl_{H_1}(\Phi(Y)) @> s_{\Phi(Y)} >> Cl_{H_2}(\varphi(\Phi(Y))) \\
 @V Cl_{H_1}(\nu) VV @VV Cl_{H_2}\varphi(\nu) V\\
 Cl_{H_1}(\Phi(X)) @> s_{\Phi(X)} >> Cl_{H_2}(\varphi(\Phi(X))),
\endCD
$$
are commutative.

\begin{theorem}\label{th:gqqsim}
Let $Var(H_1)=Var(H_2)=\Theta.$
Suppose that the geometric functors $Cl_{H_1}$ and $Cl_{H_2}\varphi$ are isomorphic by
$$
s=s(\varphi):Cl_{H_1} \to Cl_{H_2} \cdot \varphi,
$$
where the automorphism of categories $\varphi:\Theta^0\to \Theta^0$ is coordinated with the lattice structures of close congruences by additional conditions (see \cite{Pl-Sib}, \cite{Pl-Zapiski}, \cite{Pl-St}). Then the algebras $H_1$ and $H_2$ are geometrically similar.
\end{theorem}
\begin{theorem}\label{th:gqqsim1}
Suppose that the logical functors $Cl_{H_1}$ and $Cl_{H_2}\varphi$ are related by
$$
s=s(\varphi):Cl_{H_1} \to Cl_{H_2} \cdot \varphi,
$$
where the automorphism of categories $\varphi:\widetilde\Phi_\Theta\to \widetilde\Phi_\Theta $ is coordinated with the lattice structures of closed filters by additional conditions (see \cite{PZ}).
 Then the algebras $H_1$ and $H_2$ are logically similar.

\end{theorem}
Note that if there exists $\varphi$ subject to the conditions of Theorems \ref{th:gqqsim} or \ref{th:gqqsim1}, then algebras $H_1$ and $H_2$ are called {\it automorphically equivalent }(see \cite{Pl-St}, \cite{Ts1}--\cite{Ts5}, \cite{PZ1}, \cite{PZ} for definitions and discussions).

In the particular case $\varphi=id$ we have the diagrams:


  $$
\CD
Cl_{H_1}(W(Y)) @> s_{W(Y)} >> Cl_{H_2}(W(Y)) \\
 @V Cl_{H_1}(\nu) VV @VV Cl_{H_2}\varphi(\nu) V\\
 Cl_{H_1}(W(X)) @> s_{W(X)} >> Cl_{H_2}(W(X)),
\endCD
$$
and
$$
\CD
Cl_{H_1}(\Phi(Y)) @> s_{\Phi(Y)} >> Cl_{H_2}(\Phi(Y)) \\
 @V Cl_{H_1}(\nu) VV @VV Cl_{H_2}(\nu) V\\
 Cl_{H_1}(\Phi(X)) @> s_{\Phi(X)} >> Cl_{H_2}(\Phi(X)),
\endCD
$$
which imply that
\begin{corollary}\label{th:gqqsim_1}
Let $Var(H_1)=Var(H_2)=\Theta.$
Suppose that the geometric (logical) functors $Cl_{H_1}$ and $Cl_{H_2}$ are isomorphic.
Then the algebras $H_1$ and $H_2$ are geometrically (logically) similar.
\end{corollary}

This fact hints the following definition.

\begin{defn} Algebras $H_1$ and $H_2$ are called semi-$LG$ equivalent if the corresponding closure functors $Cl_{H_1}$ and $Cl_{H_2}$ are isomorphic.
\end{defn}


Now we shall formulate several problems related to  logical geometry. Let us start with the  case when $\Theta=Grp$.

\begin{problem}\label{pr:lg1} It is known \cite{Sklinos_1}, \cite{Zhitom_types}, that any group $H$ which is $LG$-equivalent to a free group $W(X)$,  is isomorphic to it. What is the situation, if $H$ is semi-$LG$ equivalent to $W(X)$.
\end{problem}

\begin{problem}\label{pr:lg2} What can be said about a group $H$ which is logically similar to a free group $W(X)$.
\end{problem}

\begin{problem}\label{pr:lg3} If two groups are $LG$-equivalent, then they are isotypic and, hence, elementary equivalent. What is the relation between the elementary equivalence of groups and their logical similarity?
\end{problem}

\begin{problem}\label{pr:lg7} Are their  logically similar groups $H_1$ i $H_2$, such that the functors  $Cl_{H_1}$ and $Cl_{H_2}\varphi$ are not isomorphic for any automorphism $\varphi$.
\end{problem}

The next problem deals with logical invariants associated with semi-$LG$-equivalence.

\begin{problem}\label{pr:lg6}

Propositions \ref{pr:imp} and \ref{pr:dis} provide implicative and disjunctive criteria for algebras to be logically equivalent. Find criteria which provide semi-$LG$-equivalence of algebras.
\end{problem}

As it was said above, the group of automorphisms of the category $\Theta^0$  plays an exceptional role in  problems related to geometrical similarity. The following problems are directed to find out what is the situation in the case of logical geometry.

\begin{problem}\label{pr:lg3} Study the group of automorphisms of the category $\widetilde\Phi_\Theta.$
\end{problem}

\begin{problem}\label{pr:lg4}Study the group of automorphisms $Aut(End(\Phi(X)))$.
\end{problem}


\subsection{Logically perfect and logically regular varieties}\label{perfect}

Up to now  we assumed that the variety $\Theta$ is arbitrary.  Now we distinguish the classes of varieties which are characterized by specific logical properties.

Let $H$ be an algebra in $\Theta$.

\begin{defn}\label{dfn:lh}
Algebra $H$ is called logically
homogeneous if for every two points $\mu: W(X)\to H$ and $\nu: W(X)\to H$ the equality $LKer(\mu)=LKer(\nu)$ holds if and only if there exists an automorphism $\sigma$ of the algebra $H$ such that $\mu=\nu\sigma$.
\end{defn}

\begin{defn}\label{dfn:lp} A variety of algebras $\Theta$ is called logically perfect if every finitely generated free in $\Theta$ algebra $W(X)$, $X\in \Gamma$ is logically homogeneous.
\end{defn}

%

\begin{defn} \label{dfn:ls}An algebra $H$ in $\Theta$ is called logically separable, if  every $H'\in\Theta$ which is $LG$-equivalent 
 to $H$ is isomorphic to $H$. 
\end{defn}

\begin{defn}\label{dfn:lr} A variety $\Theta$ is called logically regular if every free in $\Theta$ algebra $W(X)$, $X\in \Gamma$ is logically separable.
\end{defn}
\noindent


\medskip

The following theorem is valid:

\begin{theorem}\label{thm:lplr}
 {
If the variety $\Theta$ is  logically perfect, then it is logically regular.}
\end{theorem}

\begin{proof}
Let the variety $\Theta$ be  logically perfect and $W=W(X)$ a free in $\Theta$ algebra of the rang $n$, $X=\{x_1, \ldots, x_n\}$. Rewrite $W=H=<a_1, \ldots, a_n>$, where $a_1, \ldots, a_n$ are free generators in $H$. Let $H$ and $G \in \Theta$ be isotypic.

Take $\mu: W(X) \to H$ with $\mu(x_i)=a_i$. We have $\nu:W(X) \to G$ with $T^H_P(\mu)=T^G_P(\nu)$, $\nu(x_i)=b_i$, $B=<b_1, \ldots, b_n>$. The algebras $H$ and $B$ are isomorphic by the isomorphism $a_i \to b_i$, $i=1, \ldots, n$.

Indeed, $T^H_P(\mu)=T^G_P(\nu)$ implies $LKer(\mu)=LKer(\nu)$ and, hence, $Ker(\mu)=Ker(\nu)$. This gives the needed isomorphism $H \to B$.

Let us prove that $B=G$. Let $B\neq G$ and there is a $b\in G$ which doesn't lie in $B$.

Take a subalgebra $B'=<b,b_1, \ldots, b_n>$ in $G$ and a collection of variables $Y=\{y,x_1, \ldots, x_n\}$ with $\nu ': W(Y) \to G$, $\nu '(y)=b$, $\nu '(x_i)=\nu(x_i)=b_i$, $i=1, \ldots, n$.

We have $\mu ':W(Y) \to H$ with $T^H_P(\mu')=T^G_P(\nu')$. Let $\mu '(y)=a'$, $\mu '(x_i)={a'}_i$, $i=1, \ldots, n$. Let the algebras $H'=<a',{a'}_1, \ldots, {a'}_n>$ and $B'=<b,b_1, \ldots, b_n>$ be isomorphic.

Further we work with the equality $LKer(\mu ')=LKer(\nu ')$. Take a formula $u \in LKer(\mu)$ and pass to a formula $u' = (y \equiv y)\wedge u$. The point $(b_1, \ldots, b_n)$ satisfies the formula $u$ and, hence, the point $\nu '$ satisfies $u'$. Therefore, the point $\mu '$ satisfies $u'$ as well, and $u' \in LKer(\mu ')$.

Take now a point $\mu '': W(X) \to H$ setting $\mu ''(x_i)={a'}_i$, $i=1, \ldots, n$. The point $\mu'$ satisfies the formula $u'$ if and only if the point $\mu ''$ satisfies $u$. Hence, $LKer(\mu)=LKer(\mu '')$. Therefore, the point $\mu ''$ is conjugated with the point $\mu$ by some isomorphism $\sigma$.  Thus, the point $<{a'}_1, \ldots, {a'}_n>$ is a basis in $H$ and $a' \in <{a'}_1, \ldots, {a'}_n>$. This contradicts with $b \not\in <b_1, \ldots, b_n>$. So, $B=G$ and $H$ and $G$ are isomorphic.

\end{proof}

\begin{problem}\label{pr:1}  {Is the converse statement true? That is, whether every logically regular algebra is logically perfect? }
\end{problem}


It seems to us that the answer can be negative and the logical regularity of a variety $\Theta$ doesn't imply its logical perfectness. This leads to the problem

\begin{problem}\label{pr:2}  {
Find a logically regular but not logically perfect
variety $\Theta$.
In particular, consider this problem for different varieties of groups and varieties of semigroups.}
\end{problem}

Let us give some examples of perfectness and regularity for the varieties of groups and semigroups (see \cite{Houcine}, \cite{PerinSklinos}, \cite{Pillay}, \cite{Sklinos_1}, \cite{Zhitom_types}).

\begin{itemize}
\item{} The variety of all groups is logically perfect, and, hence, is logically regular.
\item{} The variety of abelian groups is logically perfect, and, hence, is logically regular.
\item{} The variety of all  nilpotent groups of  class $n$ is logically perfect, and, hence, is logically regular.
\item{} The variety of all semigroups is logically regular
\item{} The variety of all inverse semigroups is logically regular.
\end{itemize}
Now we can specify Problem \ref{pr:2}  to the case of semigroups.
\begin{problem}\label{pr:3}  Check whether the variety of all semigroups and of all inverse semigroups are logically perfect?
\end{problem}


We shall emphasize two following problems regarding solvable groups.

\begin{problem}\label{pr:4}
 {
What can be said about logical regularity and logical perfectness for the variety of all solvable groups of the derived length $n$.}
\end{problem}

\begin{problem}\label{pr:5}  {Is the variety of metabelian groups logically perfect?  Is the variety of metabelian groups logically regular?}
\end{problem}

The situation with logical regularity and logical perfectness of other varieties of algebras is not clear. Let us point out  some questions which appear by varying the variety $\Theta$. First of all:

\begin{problem}\label{pr:6} {
Let $\Theta$ be a classical variety $Com-P$, a variety of commutative and associative algebras with unit over a field $P$. The problem is to verify its logical regularity and logical perfectness.}
\end{problem}
The same question stands with respect to some other  well-known varieties. So, are the following varieties logically perfect or logically regular.

\begin{problem}\label{pr:7} {
The variety $Ass-P$ of associative algebras over a field $P$.}
\end{problem}

\begin{problem}\label{pr:8} {
The variety $Lee-P$ of Lee algebras over a field $P$.}
\end{problem}

\begin{problem}\label{pr:9} {
The variety of $n$-nilpotent associative algebras.}
\end{problem}

\begin{problem}\label{pr:10} {
The variety of $n$-nilpotent Lee algebras.}
\end{problem}

\begin{problem}\label{pr:11} {
The varieties of solvable Lee/associative algebras of the derived length $n$ .}
\end{problem}

It is also important to find out how the passage from a semigroup/group to a semigroup/group algebra behaves with respect to logical regularity and logical perfectness. This  leads to the problem:
\begin{problem}\label{pr:11} {
Let $S$ be a semigroup/group and $P$ a field, both logically homogeneous. Whether it is true that the semigroup/group  algebra $PS$ is logically homogeneous as well?}
\end{problem}

\subsection{Logically noetherian and saturated algebras }\label{sub:mt}

\begin{defn} {
An algebra $H$ is called logically noetherian if for any set of formulas $T \subset \Phi(X)$, $X \in \Gamma$ there is a finite subset $T_0$ in $T$ determining the same set of points $A$ that is determined by the set $T$.}
\end{defn}

\begin{defn} {
An algebra $H \in \Theta$ is called $LG$-saturated if for every $X \in \Gamma$,  each ultrafilter $T$ in $\Phi(X)$ containing $Th^X(h)$ has the form $T = LKer(\mu)$ for some $u : W(X) \to H$.}
\end{defn}

\begin{theorem} {
If an algebra $H$ is logically noetherian then $H$ is LG-saturated.}
\end{theorem}



\begin{proof}
We start from the homomorphism:
$$
Val_H^X:\Phi(X)\to Hal_\Theta^X(H).
$$
Here $Ker(Val_H^X)=Th^X(H)$. Consider the quotient algebra $\Phi(X)/Th^X(H)$ which is isomorphic to a subalgebra in $Hal_\Theta^X(H)$. For every $u\in \Phi(X)$ denote by $[u]$ the image of $u$ in the quotient algebra. By definition $[u]=0$ means that $Val_H^X(u)$ is the empty subset in $Hom(W(X),H)$. Analogously $[u]=1$ means that  $Val_H^X(u)$ is the whole space $Hom(W(X),H)$ and, thus, $u\in Th^X(H)$.

Denote by $T$  an ultrafilter in $\Phi(X)$, containing the theory $Th^X(H)$. We need to check that there is a point $\mu:W(X) \to H$ such that $T=LKer(\mu)$. Let $[u]=0$. Then $[\neg u]=1$, which means that $\neg u\in Th^X(H)\subset T$. Hence $\neg \in T$. Then $u$ does not belong to $Th^X(H),$ since $T$ cannot contain both $u$ and $\neg u$. So $u\notin T$. Thus, if $[u]=0$ then $u\notin T$. If $u\in T$, then $[u]\neq 0$. This means that $Val^X_H(u)$ is not empty. Thus, we have a point $\mu:W(X)\to H$ which satisfies $u$, that is $u\in LKer(\mu)$. Since $H$ is logically noetherian then there exists a finite subset $T_0=\{u_1,\ldots,u_n\}$ such that $T^L_H={(T_0)}^L_H$. Take $u=u_1\wedge u_2\wedge\ldots u_n.$ Since all $u_i\in T$ then $u\in T$, Then there exists $\mu$ satisfying formula $u$. The same point $\mu$ satisfies every $u_i$. Hence $\mu\in (T_0)^L_(H)$=$T^L_(H)$ and $T$ lies in $LKer(\mu)$.  Therefore $T=LKer(\mu)$.

\end{proof}

Each finite algebra $H$ is logically noetherian. Therefore, every finite $H$ is saturated. This holds for every $\Theta$.





\subsection{Automorphically finitary algebras }\label{sub:af}

We have already mentioned that the group $Aut(H)$ acts in each space $Hom(W(X),H)$, $X \in \Gamma$.

\begin{defn}\label{def:af} {
Let us call an algebra $H$ automorphically finitary if in each such action there is only finite number of $Aut(H)$-orbits.}
\end{defn}

It is easy to show that if algebra $H$ is automorphically finitary, then it is logically noetherian. The example of abelian groups of the exponent $p$ shows that there exist infinite automorphically finitary algebras and, thus, there are infinite saturated algebras.


\begin{problem}\label{pr:15} {
Describe all automorphically finitary abelian groups.}
\end{problem}

\begin{problem}\label{pr:16} {
Construct examples of non-commutative automorphically finitary  groups.}
\end{problem}

\begin{problem}\label{pr:17} {
Classify abelian groups by $LG$-equivalence relation}.
\end{problem}

Let us make some comments regarding Problem \ref{pr:17}. According to Theorem \ref{thm:lgiso}, $LG$-equivalent abelian groups are isotypic.  As we know (Corollary \ref{cor:el}), isotypeness of algebras implies their elementary equivalence. Classification of abelian groups with respect to elementary equivalence had been obtained by W.Szmielew in her classical paper \cite{Sz}. So, Problem \ref{pr:17} asks how one should modify the list from \cite{Sz} in order to obtain the isotypic abelian groups.


We had considered two important characteristics of varieties of algebras, namely, their logical perfectness and logical regularity. Let us introduce one more characteristic.

We call a variety $\Theta$ exceptional if
\begin{itemize}
\item any two free in $\Theta$  algebras $W(X)$ and $W(Y)$ of a finite rank, generating the whole $\Theta$, are elementarily equivalent, and
\item if they are isotypic then they are isomorphic.
\end{itemize}

\begin{problem}\label{pr:18} {
Whether it is true that only the variety of all groups is unique.}
\end{problem}



\section{Model theoretical types and logically geometric types}\label{sec:type}

\subsection{Definitions of types}\label{sub:ty}

The notion of a type is one of the key notions of Model Theory. In what follows we will distinguish between model theoretical types (MT-types) and logically geometric types (LG-types). Both kinds of types are oriented towards some algebra $H \in \Theta$, where $\Theta$ is a fixed variety of algebras.

Generally speaking, a type of a point $\mu : W(X) \to H$ is a logical characteristic of the point $\mu$.
Model-theoretical idea of a type and its definition is described in many sources, see, in particular, \cite{Hodges}, \cite{Marker}. We consider this idea from the perspective of algebraic logic (cf., \cite{PlAlPl}) and give all the definitions in the corresponding terms.

Proceed from the algebra of formulas $\Phi(X^0)$, where $X^0$ is an infinite set of variables. It arrives from the algebra of pure first-order formulas with equalities $w \equiv w'$, $w, w' \in W(X^0)$ by Lindenbaum-Tarski algebraization approach (cf. Section \ref{sub:ms}). 
 $\Phi(X^0)$ is an $X^0$-extended Boolean algebra, which means that  $\Phi(X^0)$ is a Boolean algebra with quantifiers $\exists x $, $x\in X^0$ and equalities $w\equiv w'$, where $w,w'\in W(X^0)$, where $W(X^0)$ is the free over $X^0$ algebra in $\Theta$. All these equalities generate the algebra $\Phi(X^0)$. Besides, the semigroup $End(W(X^0))$ acts in the Boolean algebra $\Phi(X^0)$ and we can speak of a polyadic algebra $\Phi(X^0)$. However, the elements $s \in End(W(X^0))$ and the corresponding $s_\ast$ are not included in the signature of the algebra $\Phi(X^0)$.

 Since $\Phi(X^0)$ is a one-sorted algebra,  one can speak, as usual, about free and bound occurrences of the variables in   the formulas $u \in \Phi(X^0)$.


Define further $X$-special formulas in $\Phi(X^0)$, $X=\{x_1, \ldots, x_n \}$. Take $X^0 \backslash X = Y^0$. A formula $u \in \Phi(X^0)$ is {\it $X$-special} if each its free variable is occurred in $X$ and each bound variable belongs to $Y^0$. A formula $u\in \Phi(X^0) $ is closed  if it does not have free variables. Only finite number of variables occur in each formula.

Denoting  an $X$-special formula $u$ as $u=u(x_1, \ldots , x_n$; $y_1, \ldots , y_m)$  we solely mean  that the set $X$ consists of  variables $x_i$, $i=1,\ldots n$, and those of them who occur in $u$, occur  freely.




\begin{defn}\label{def:tpMT}
Let $H$ be an algebra from $\Theta$. An $X$-type (over $H$) is a set of $X$-special formulas in $\Phi(X^0)$, consistent
 with the elementary theory of the algebra $H$.
 \end{defn}
 We  call such type an $X$-MT-type (Model Theoretic type) over $H$. An $X$-MT-type is called {\it complete} if it is maximal with respect to inclusion. Any complete $X$-MT-type is a Boolean ultrafilter in the algebra $\Phi(X^0)$. Hence, for every $X$-special formula $u\in \Phi(X^0)$, either $u$ or its negation belongs to a complete type.


 \begin{defn}\label{def:tpLG}
  An $X$-LG-type (Logically Geometric type) (over $H$) is a Boolean ultrafilter in the corresponding $\Phi(X)$, which contains the elementary theory $Th^X(H)$.
 \end{defn}

 So, any $X$-MT-type lies in the one-sorted algebra $\Phi(X^0)$. Any $X$-LG-type lies in the domain $\Phi(X)$ of the multi-sorted algebra $\widetilde \Phi$.

 We denote the MT-type of a point $\mu :W(X) \to H$  by $Tp^{H}(\mu)$, while the LG-type of the same point is, by definition, its logical kernel $LKer(\mu)$.


\begin{defn}\label{def:tip}
Let a point $\mu :W(X) \to H$, with $a_i = \mu(x_i)$,  be given.  An $X$-special formula $u = u(x_1, \ldots, x_n ; y_1, \ldots, y_m)$ belongs to the type $Tp^{H} (\mu)$ if the formula $u(a_1, \ldots, a_n ; y_1, \ldots, y_m)$ is satisfied in the algebra $H$.
\end{defn}


 The type $Tp^{H} (\mu)$ consists of all $X$-special formulas satisfied on $\mu$. It is a complete $X$-MT-type  over $H$.

By definition, the formula $v=u(a_1, \ldots, a_n ; y_1, \ldots, y_m)$ is closed. Thus, if it is satisfied one a point, then its value set $Val^H_X(v)$ is the whole $Hom(W(X),H)$.



 Note that in our definition of an $X$-MT-type the set of free variables in the formula $u$ is not necessarily the whole $X=\{x_1,\ldots,x_n\}$ and can be a part of it. In particular, the set of free variables can be empty. In this case the formula $u$ belongs to the type if it is satisfied in $H$.



In the previous sections the algebra $\widetilde \Phi$ was built basing on the set $\Gamma$ of all finite sunsets of the set $\Gamma$. In fact, one can take instead of $\Gamma$ the system
$\Gamma^\ast=\Gamma \bigcup X^0$ and construct the corresponding multi-sorted algebra. Then, to each homomorphism $s:W(X^0)\to W(X)$ it corresponds a morphism $s_\ast:\Phi(X^0)\to \Phi(X)$ and, vice versa,  $s:W(X)\to W(X^0)$ induces $s_\ast:\Phi(X)\to \Phi(X^0)$. In this setting the extended Boolean algebra $Hal_\Theta^{X^0}(H)$ and the homomorphism $Val_H^{X^0} :\Phi(X^0)\to Hal_\Theta^{X^0}(H)$ are defined in a usual way. A point $\mu: W(X^0)\to H$ satisfies $u\in \Phi(X^0)$ if $\mu \in Val_H^{X^0}(u)$.

One more  remark. Since $\Phi(X^0)$ is generated by equalities, when we say that a variable 
 occur in a formula $u\in \Phi(X^0)$, this means that it occur in one of the equalities $w=w'$, participating in $u$.
 The set of variables occurring in  $u$ determines  a subalgebra $\Phi(X \cup Y)$ in $\Phi(X^0)$, such that  $u\in \Phi(X \cup Y)$.

 If we stay in one-sorted logic, this is a subalgebra in the signature of the one-sorted algebra $\Phi(X^0)$. 

On the other hand, we can view algebra $\Phi(X \cup Y)$ as an object in the multi-sorted logic. Here, 
 to every homomorphism $s : W(X \cup Y) \to W(X' \cup Y')$ corresponds a morphism $s_\ast : \Phi(X \cup Y) \to \Phi(X' \cup Y')$. For $u \in \Phi(X \cup Y)$ we have $s_\ast u \in \Phi(X' \cup Y')$. Let $u$ be an $X$-special formula. It is important to know for which $s$   the formula $s_\ast u$ is $X'$-special. 

\subsection{Another characteristic of the type $Tp^{H} (\mu)$}

We would like to relate an MT-type of a point to its LG-type.

Consider a special homomorphism $s: W(X^0) \to W(X)$ for an infinite set $X^0$ and its finite subset $X=\{x_1, \ldots , x_n\}$, such that %
$s(x)=x$ for each $x \in X$, i.e., $s$ is identical on the set $X$. According to the transition from $s$ to $s_\ast$, we have $$s_\ast :\Phi(X^0) \to \Phi(X).$$

\begin{theorem}\label{th:crit}
For each special homomorphism $s$, each special formula $u=u(x_1, \ldots, x_n; y_1, \ldots , y_m)$ in $\Phi(X^0)$ and every point $\mu : W(X) \to H$, we have $u \in Tp^H(\mu)$ if and only if $s_\ast u \in LKer(\mu)$. Here, in the first case $u$ is considered in one-sorted algebra $\Phi(X^0)$, while in the second case $s_\ast u$ lies in the domain $\Phi(X)$ of the multi-sorted $\widetilde\Phi=(\Phi(X), \ X\in \Gamma^\ast.)$
\end{theorem}


\begin{proof}
We need one more look at a formula $u \in Tp^H(\mu)$. Given a point $\mu$, consider a set $A_\mu: W(X)\to H$ of the points $\eta : W(X^0) \to H$ defined by the rule $\eta(x_i) = \mu(x_i)=a_i$ for $x_i \in X$ and, $\eta(y)$ is an arbitrary element in $H$ for $y \in Y^0$. Denote
$$T_\mu = \bigcap_{\eta \in A_\mu} LKer(\eta).$$
Here, as usual, $LKer(\eta)$ is the ultrafilter in $\Phi(X^0)$, consisting of formulas $u$ valid on a point $\eta$.
It is proved \cite{PlAlPl}, that a special formula $u$ belongs to the type $Tp^H(\mu)$ if and only if $u \in T_\mu$, which is equivalent to $Val^{X^0}_H (u) \supset A_\mu$.

Note that the formula $u$ of the kind
$$x_1 \equiv x_1 \wedge \ldots \wedge x_n \equiv x_n \wedge v(y_1, \ldots , y_m)$$
belongs to each $LKer(\eta)$ if the closed formula $v(y_1, \ldots , y_m)$ is satisfied in the algebra $H$. This means also that $T_\mu$ is not empty for every $\mu$.

Return to the special homomorphism $s: W(X^0)\to W(X)$ and consider the point $\mu s : W(X^0) \to H$. For $x_i\in X$ we have $\mu s(x_i)=\mu(x_i)=a_i$. Hence, the point  $\mu s $ belongs to $A_\mu$.


Observe  that for the formula $u = u(x_1, \ldots, x_n ; y_1, \ldots, y_m),$ the formula $u(a_1, \ldots, a_n ; y_1, \ldots, y_m)$ is satisfied in the algebra $H$ if the set $A_\mu$ lies in $Val^{X^0}_H (u)$. Thus, $\mu s $ belongs to $Val^{X^0}_H (u)$. By definition of $s_\ast$ we have that $\mu$ lies in $s_\ast Val^{X^0}_H (u)=Val^{X}_H(s_\ast u)$, which means that
$$
s_\ast u \in LKer(\mu).
$$
We proved the statement in one direction.


Conversely, let $s_\ast u \in LKer(\mu)$. Then
$$\mu \in Val^{X}_H(s_\ast u)=s_\ast Val^{X^0}_H (u)$$
and $\mu  s\subset  Val^{X^0}_H (u)$.
Since the formula $u(a_1, \ldots, a_n ; y_1, \ldots, y_m)$ is satisfied in $H$, then every point from the set $A_\mu$ belongs to $Val^{X^0}_H(u)$ (see also \cite{CK}). This means that the formula $u$ belongs to $Tp^H(\mu)$.




\end{proof}

Recall that we have mentioned the notion of a saturated algebra. It was LG-saturation. In the Model Theory MT-saturation is defined. MT-saturation of the algebra $H$ means that for any $X$-type $T$ there is a point $\mu : W(X) \to H$ such that $T \subset Tp^H(\mu)$.

\begin{theorem}
If algebra $H$ is LG-saturated then it is MT-saturated.
\end{theorem}

\begin{proof}
Let algebra $H$ be LG-saturated and $T$ be $X$-MT-type correlated with $Th^{X^0}(H)$. We can assume that the theory $Th^{X^0}(H)$ is contained in the set of formulas $T$.

Take a special homomorphism $s : W(X^0) \to W(X)$ and pass to $s_\ast : \Phi(X^0) \to \Phi(X)$. Take a formula $s_\ast u \in \Phi(X)$ for each formula $u \in T$ and denote the set of all such $s_\ast u$ by $s_\ast T$. This set is a filter in $\Phi(X)$ containing the elementary theory $Th^X(H)$, since, if $u \in Th^{X^0}(H)$ then $s_\ast u \in Th^{X}(H)$.

Further we embed the filter $s_\ast T$ into the ultrafilter $T_0$ in $\Phi(X)$ which contains the theory $Th^{X}(H)$. By the LG-saturation of the algebra $H$ condition, $T_0 = LKer(\mu)$ for some point $\mu : W(X) \to H$. Thus, $s_\ast u \in LKer (\mu)$ for each formula $u \in T$. Hence (Theorem \ref{th:crit}), $u \in Tp^H(\mu)$ for each $u \in T$, and $T \subset Tp^H(\mu)$. This gives MT-saturation of the algebra $H$.
\end{proof}

We do not know whether MT-saturation implies LG-saturation. It seems that not. If it is the case, then LG-saturation of an algebra $H$ is stronger than its MT-saturation.

\subsection{Correspondence between $u\in\Phi(X)$ and  $\widetilde u\in\Phi(X^0)$ }\label{sec:uvolna}

\begin{defn}
 A formula $u\in \Phi(X)$ is called correct, if there exists an $X$-special formula $\widetilde u$ in $\Phi(X^0)$ such that  for every point $\mu: W(X)\to H$ we have $u\in LKer \mu$ if and only if $\widetilde u\in T^H_p(\mu)$.
\end{defn}


Now, for the sake of completeness and  for the aims of clarity
we give a proof of the principal Theorem \ref{thm:zhi}
of G.Zhitomiskii  (see \cite{Zhitom_types} for the original exposition). This fact will be  essentially used in Theorem \ref{thm:zh} and Theorem \ref{thm:eq}. We hope this will help to 
reveal ties between two approaches to the idea of a type of a point: the one-sorted model theoretic approach and the multi-sorted logically-geometric approach. Note that the proofs are  sometimes  different from that of \cite{Zhitom_types}.

\begin{theorem} \label{thm:zhi}\cite{Zhitom_types} For every $X=\{x_1,\ldots,x_n\}$, every formula $u\in \Phi(X)$ is correct.
\end{theorem}

\begin{proof} First of all, each equality $w=w'$, $w,w'\in W(X)$ is a correct formula. This follows from
$\widetilde{(w=w')}=(w=w')$.

Take two correct  formulas $u$ and $v$, both from $\Phi(X)$.  Show that  $u\wedge v$, $u\vee v$ and $\neg u$ are also correct. We have $\widetilde u$ and $\widetilde v$. Define
$$\widetilde {u\wedge v}=\widetilde u \wedge \widetilde v,$$
$$\widetilde {u\vee v}=\widetilde u \vee \widetilde v,$$
$$\widetilde {\neg u}=\neg\widetilde u.$$

By definition,  we have $u\in LKer \mu$ if and only if $\widetilde u\in T^H_p(\mu)$ for every point $\mu: W(X)\to H$. The same is true with respect to $v$ and $\neg u$. Let $u\vee v\in LKer \mu$ and, say, $u\in LKer \mu$. Then $\widetilde u\in T^H_p(\mu)$, and, hence, $\widetilde u \vee \widetilde v=\widetilde {u\vee v}\in T^H_p(\mu)$. Conversely, let $\widetilde {u\vee v}=\widetilde u \vee \widetilde v\in T^H_p(\mu).$ Suppose that $\widetilde u\in T^H_p(\mu)$. Then $u\in LKer \mu$, that is $u\vee v\in LKer \mu$. The similar proofs work for the correctness of the formulas $u\wedge v$  and $\neg u$. In the latter case one should use the completeness property  of a type: ${\neg u}\in T^H_p(\mu)$ if and only if $u \notin T^H_p(\mu)$.


Our next aim is to check that if the formula $u\in \Phi(X)$ is correct, then the formula $\exists xu\in \Phi(X)$ is also correct.

Beforehand, note that it is hard to define free and bounded variables in the algebra $\Phi(X)$. This is because of the
multi-sorted nature of $\Phi(X)$ and the presence in it of the formulas which include operations of the type $s_*$.
So, the syntactical definition of  $\exists xu\in\Phi(X)$ is a sort of problem and we will proceed from the semantical definition of this formula.

Namely, a point $\mu: W(X)\to H$ satisfies the formula  $\exists xu\in\Phi(X)$ if there exits a point $\nu: W(X)\to H$ such that $u\in LKer(\nu)$ and $\mu$ coincides with $\nu$ for every variable $x'\neq x$, $x'\in X$.

 Indeed, a point $\mu : W(X)\to H$ satisfies $\exists xu\in\Phi(X)$ if $\mu\in Val^X_H(\exists xu)=\exists x(Val^X_H(u))$ (see Subsection \ref{sub:qq}).
Denote the set $Val^X_H(u)$ in $Hal_\Theta^X(H)=Bool(W(X),H)$ by $A$. Then $\mu$ belongs to $\exists xA$. Using the definition of existential quantifiers in $Hal_\Theta^X(H)$ (Subsection \ref{ex:im}) and the fact that $u\in LKer(\nu)$  if and only if $\nu \in Val^X_H(u)$,  we arrive to the definition above.


 Since $u$ is correct, there exists an $X$-special formula $\widetilde u \in \Phi(X^0)$,
 $$\widetilde u =\widetilde{u}(x_1,\ldots,x_n, y_1, \ldots, y_m), \ x_i\in X,\  y_i\in Y^0=(X^0\setminus X),$$  such that
$\widetilde u\in T^H_p(\mu)$ if and only if $u\in LKer(\mu)$, where  $\mu: W(X)\to H$.

  Define

$$
\widetilde {\exists x u}=\exists x\widetilde u.
$$
The formula $\exists x\widetilde u$ is not $X$-special since $x$ is bound (we assume that $x$ coincides with one of $x_i$, say $x_n$). Take a variable $y\in X^0$, such that $y$ is different from each $x_i\in X$ and $y_j\in Y^0$.

Define  $\exists y\widetilde{u}_y$ to be a formula which coincides with $\exists x\widetilde u$ modulo replacement of $x$ by $y$.
 So, $\exists y\widetilde{u}_y$ has one less free variable and one more bound variable than $\exists x\widetilde u$.

Consider endomorphism $s$ of $W(X^0)$ taking $s(x)$ to $y$ and leaving all other variables from $X^0$ unchanged. Let $s_*$ be the corresponding automorphism of the one-sorted Halmos algebra $\Phi(X^0)$. Then $s_*(\exists x\widetilde u)=\exists s_*(x) s_*(\widetilde u)=\exists y\widetilde{u}_y$.


Define
$$
\widetilde {\exists x u}=\exists y\widetilde{u}_y.
$$

Thus, in order to check that $\exists xu$ is correct, we need to verify that
 for every $\mu: W(X)\to H$ the formula $\exists xu$ lies in $LKer (\mu)$ if and only if $\exists y\widetilde{u}_y\in T^H_p(\mu)$.

Let $\exists xu$ lies in $LKer (\mu)$. Thus, there exits a point $\nu: W(X)\to H$ such that $u\in LKer(\nu)$ and $\mu$ coincides with $\nu$ for every variable $x'\neq x$, $x'\in X$.
  Consider $X_y=\{x_1,\ldots, x_{n-1}, y\}$.


  We have points $\mu:W(X)\to H$, $\mu': X_y\to H$ where $\mu'(x_i)=\mu(x_i)=a_i$, and
  $\mu'(y)$ is an arbitrary element $b$ in $H$. We have also $\nu: W(X)\to H$ and $\nu': X_y\to H$, where $\nu'(x_i)=\nu(x_i)$, and
  $\nu'(y)=\nu(x_n)$. So, $\nu$ and $\nu'$ have the same images. Denote it $(a_1,a_2,\ldots, a_{n-1},a_n)$, $a_i\in H$, i.e., $\nu'(y)=a_n$.

  Take
  $$\widetilde u_y =\widetilde{u}(x_1,\ldots,x_{n-1},y, y_1, \ldots, y_m), $$
  Since the formula $\exists y\widetilde{u}(a_1,\ldots,a_{n-1},b, y_1, \ldots, y_m) $ is closed for any $b$, then either it is satisfied on any point $\mu'$, or no one of   $\mu'$ satisfies this formula. We can take $b=a_n$, that is $\mu'=\nu'$. Since $\nu$ and $\nu'$ have the same images, and $u$ is correct, the point $\nu'$ satisfies $\widetilde{u}_y$. Then $\nu'$ satisfies  $\exists y\widetilde{u}_y$. Hence $\exists y\widetilde{u}(x_1,\ldots,x_{n-1},y, y_1, \ldots, y_m) $ is satisfied on $\mu'$ for any $b$. This means that $\exists y\widetilde{u}_y\in T^H_p(\mu')$ for every $\mu'$.
  We can take $\mu'$ to be $\mu$. Then $\widetilde{\exists x u}\in T^H_p(\mu).$

  Conversely, let  $\widetilde{\exists x u}\in T^H_p(\mu)$.
    Take a point $\nu: W(X)\to H$ such that $\nu(x_i)=\mu(x_i)$,   $i=1,\ldots,{n-1}$, $\nu(x_n)=\nu(y)$. We have $\widetilde u\in T^H_p(\nu)$. Since $\widetilde u$ is correct, then $u$ in $LKer(\nu)$. The points $\mu$ and $\nu$ coincide on all $x_i$, $i\neq n$. Thus $\exists u$ belongs to  $LKer(\mu)$.

    \medskip


    \medskip

    It remains to check that the operation $s_*$ respects correctness of formulas. Let $X=\{x_1,\ldots,x_n\}$, $Y=\{y_1,\ldots,y_m\}$, and a morphism $s: W(Y)\to W(X)$ be given. Take the corresponding $s_*:\Phi(Y)\to \Phi(X)$.  Given $v\in \Phi(Y)$ consider $u=s_*v$ in $\Phi(X)$. We shall show that if $v$ is $Y$-correct then $u$ is $X$-correct.

   We have $u\in LKer(\mu)$ , $\mu:W(X)\to H$ if and only if $v\in LKer(\nu)$, $\nu:W(Y)\to H$ for $\mu s=\nu$.
   Indeed, $u=s_*v\in LKer(\mu)$ means that $\mu\in Val^X_H(s_* v)=s_*Val^Y_H(v)$ and thus, $\mu s \in Val^Y_H(v)$. Hence,  for $\nu=\mu s$ we have $v \in LKer(\nu)$. Conversely, let $v \in LKer (\nu)$ and $\mu s = \nu\in Val_H^Y(v)$. We have $\mu \in s_*Val_H^Y(v)=Val_H^X(s_*v)=Val_H^X(u)$ and $u\in LKer (\mu)$.

    Note that morphism $s_*:\Phi(Y)\to \Phi(X)$ is a homomorphism of Boolean algebras. Suppose that $v\in \Phi(Y)$ is correct. We have
    $$\widetilde v= \widetilde{v}(y_1,\ldots, y_m, z_1,\ldots,z_t),$$
    where all $z_i$ are bound and belong to $Z=\{z_1,\ldots,z_t\}$. All free variables in $\widetilde v$ belong to $Y$ (it is assumed that not necessarily all variables from $Y$ occurs in $\widetilde v$). In this sense $\widetilde v$ is $Y$-special.

    We will define also the formula $\widetilde u$ and show that in our situation $ \widetilde u\in Tp^H(\mu)$ if and only if $ \widetilde v\in Tp^H(\nu).$

    Consider $Z'=\{z'_1,\ldots, z'_t\}$, where all $z'_i$ do not belong to $X$. Take the free algebras $W(X\cup Z')$ and $W(Y\cup Z)$. Define homomorphism $s':W(Y\cup Z)\to W(X\cup Z')$ extending $s:W(Y)\to W(X)$ by $s'(z_i)=z'_i.$ The commutative diagram of homomorphisms takes place:

$$
\CD
W(Y\cup Z) @> s' >> W(X\cup Z')\\
@V  s^1 VV @VV s^2 V\\
W(Y) @>s>> W(X).
\endCD
$$

Here $s^1$ and  $s^2$ are special homomorphisms which act identically on $Y$ and $X$, respectively. The corresponding commutative diagram of morphisms of algebras of formulas is as follows:
$$
\CD
\Phi(Y\cup Z) @> {s'}_\ast >> \Phi(X\cup Z')\\
@V  s^1_\ast VV @VV s^2_\ast V\\
\Phi(Y) @>s_\ast>> \Phi(X).
\endCD
$$
This diagram is commutative due to the fact that the product of morphisms of algebras of formulas corresponds to the product of homomorphisms of free algebras. Apply the diagram to $Y$-special formula $\widetilde v$ which belongs to the algebra $\Phi(Y\cup Z)$. Then, $s^2_\ast {s'}_\ast \widetilde v = s_\ast s^1_\ast \widetilde v$. Assume that $\widetilde u = {s'}_\ast \widetilde v$. Here, $\widetilde u$ is an $X$-special formula, contained in the algebra $\Phi(X\cup Z')$. We need to prove that for any point $\mu:W(X) \to H$ the inclusion $\widetilde u \in Tp^H (\mu)$ holds if and only if $u \in LKer(\mu)$.

We use the criterion from Section \ref{sec:type} (Theorem \ref{th:crit}): $\widetilde u \in Tp^H (\mu)$ if and only if $s^2_\ast \widetilde u \in LKer(\mu)$. Let us prove the latter inclusion. The similar criterion is valid for the formula $\widetilde v$. Since the formula $v$ is correct, then $\widetilde v \in Tp^H (\nu)$, where $\nu = \mu s$. Hence, $s^1_\ast \widetilde v \in LKer(\nu)$, which means that the point $\nu$ belongs to the set $Val^Y_H(s^1_\ast \widetilde v)$. Since $\nu = \mu s$, then $\mu \in Val^X_H(s_\ast s^1_\ast \widetilde v) = Val^X_H(s^2_\ast {s'}_\ast \widetilde v) = Val^X_H(s^2_\ast \widetilde u)$. This leads to the inclusion $s^2_\ast  \widetilde u \in LKer(\mu)$, which gives $\widetilde u \in Tp^H (\mu)$.

The same reasoning in the opposite direction shows that the inclusion $\widetilde u \in Tp^H (\mu)$ is equivalent to that of $\widetilde v \in Tp^H (\nu)$.

It is worth to recall that we started from the fact $u \in LKer(\mu)$ if and only if $v \in LKer(\nu)$. But, $v \in LKer(\nu)$ because of the correctness of the formula $v$. Thus, $u \in LKer(\mu)$. Hence, the transition from $u$ to $\widetilde u$ guarantees the correctness of the formula $u$.

  Hence, the set of all correct $X$-formulas, for various $X$, respects all operations of the multi-sorted algebra $\widetilde \Phi$. Since $\widetilde \Phi$ is generated by equalities, which are correct, the subalgebra of all correct formulas in $\widetilde \Phi$ coincides with $\widetilde \Phi$. Thus every $u\in \widetilde \Phi(X)$ for every $X$, is correct.

\end{proof}

Consider a simple example illustrating the action of $s_\ast$. Take $Y=\{y_1,y_2\}$ and $X=\{x_1,x_2,x_3\}$ and let $s$ be a homomorphism $s:W(Y)\to W(X)$. Take also variables $z$ and $z'$ and extend $s$ to $s':W(Y \cup z)\to W(X\cup z')$ assuming $s'(z)=z'$.
We have also morphism $s'_\ast:\Phi(Y\cup z)\to \Phi(X\cup z')$. Take an equality $w(y_1,y_2,z)\equiv w'(y_1,y_2,z)$ in $\Phi(Y\cup Z)$.  Consider $\exists z(w\equiv w')$  and apply $s'_\ast$. We have
$$
s'_\ast(\exists z(w\equiv w'))=\exists z'(s'w\equiv s'w').
$$
Here $s' w =s'(w(y_1,y_2,z)=w(w_1,w_2,z'),$
where $w_i=s(y_1)=w_i(x_1,x_2,x_3)$ and

$$s'_*(\exists z(w\equiv w')=\exists z'(w(w_1(x_1,x_2,x_3), w_2(x_1,x_2,x_3),z')$$

$$\equiv w'( w_1(x_1,x_2,x_3), w_2(x_1,x_2,x_3),z')).$$

\subsection{LG and MT-isotypeness of algebras}\label{sec:isotyp}

The following 
theorem 
helps to clarify the notion of isotypeness of algebras.


\begin{theorem} \label{thm:zh} \cite{Zhitom_types} Let the points $\mu:W(X)\to H_1$ and $\nu:W(X)\to H_2$ be given. Then
$$
Tp^{H_1}(\mu)=Tp^{H_2}(\nu)
$$
if and only if
$$
LKer(\mu)=LKer(\nu).
$$
\end{theorem}

\begin{proof} Let the points $\mu:W(X)\to H_1$ and $\mu:W(X)\to H_2$ be given and let $
Tp^{H_1}(\mu)=Tp^{H_2}(\mu).$ Take $u\in LKer(\mu)$. Then $\widetilde u \in Tp^{H_1}(\mu)$ and, thus, $\widetilde u \in Tp^{H_2}(\nu). $ Hence, $u\in LKer(\nu)$. The same is true in the opposite direction.

Let, conversely, $LKer(\mu)=LKer(\nu).$ Take an arbitrary $X$-special formula $u$ in $Tp^{H_1}(\mu)$. Take a special homomorphism from $s:W(X^0)\to W(X)$. It corresponds the morphism $s_*:\Phi(X^0)\to\Phi(X).$ Then, using Theorem \ref{th:crit}, the formula $u\in Tp^H(\mu)$ if and only if $s_\ast  u \in LKer(\mu).$  Then $s_\ast  u \in LKer(\nu).$ Then $u\in Tp^H(\nu)$.
\end{proof}


\begin{defn}\label{defn:tp}
Given $X$, denote by $S^X(H)$ the set of MT-types of an algebra $H$, implemented (realized) by points in $H$. Algebras $H_1$ and  $H_2$ are called MT-isotypic if $S^X(H_1)=S^X(H_2)$ for any $X \in \Gamma$.
\end{defn}

Theorem \ref{thm:zh}  implies

\begin{corollary}\label{cor:lg}
Algebras $H_1$ and $H_2$ in the variety $\Theta$ are MT-isotypic if and only if they are LG-isotypic.
\end{corollary}

So, it doesn't matter which type (LG-type or MT-type) is used in the definition of isotypeness. Hence, by Theorem \ref{thm:lgiso}, algebras $H_1$ and $H_2$ in the variety $\Theta$ are MT-isotypic if and only if they are LG-equivalent.


If algebras $H_1$ and  $H_2$ are isotypic then they are locally isomorphic. This means that if $A$ is a finitely generated subalgebra in $H$, then there exists a subalgebra $B$ in $H_2$ which is isomorphic to $A$ and, similarly, in the direction from $H_2$ to $H_1$.

On the other hand, local isomorphism of $H_1$ and  $H_2$ does not imply their isotypeness: the groups $F_n$ and $F_m$, $m,n > 1$ are locally isomorphic, but they are isotypic only for $n=m$.

Isotypeness imply elementary equivalence of algebras, but the same example with $F_n$ and $F_m$ shows that the opposite is wrong.

In Section 2 we pointed out several problems related to isotypic algebras. Let us give some other problems:

\begin{problem} {
Let $H_1$ and $H_2$ be two finitely generated isotypic algebras. Are they always isomorphic?}
\end{problem}

In particular:

\begin{problem} {
Let $G_1$ and $G_2$ be two finitely generated isotypic groups.  Are they always isomorphic?}
\end{problem}

\begin{problem} {
Let $H_1$ be a finitely generated algebra and $H_2$ is an isotypic to it algebra. Is $H_2$ also finitely generated?}
\end{problem}

The next problem is connected with the previously named problems on
isotypeness and isomorphism of free algebras.

\begin{problem}\label{pr:31}
Let two isotypic finitely-generated free algebras $H_1$ and $H_2$ and two points $\mu: W(X)\to H_1$ and
$\nu: W(X)\to H_2$
be given. Let $LKer(\mu)=LKer(\nu)$. Is it true that there exists an isomorphism $\sigma: H_1\to H_2$ such that $\mu \sigma = \nu$?
\end{problem}

\subsection{LG and MT-geometries}\label{sub:mt}

Compare, first, different approaches to the notion of a definable set in the affine space $Hom(W(X),H)$.

 Suppose that  the variety of algebra $\Theta$, an algebra $H\in\Theta$ and the finite set $X=\{x_1,\ldots,x_n\}$ are fixed.

In the affine space $Hom(W(X),H)$ consider subsets $A$, whose points have the form $\mu:W(X)\to H$. Each point
$\mu:W(X)\to H$ has a classical kernel $Ker(\mu)$, a logical kernel $LKer(\mu)$ and a type $(Tp^H(\mu)$). Correspondingly, we have three different geometries: algebraic geometry ($AG$), logical geometry ($LG$), and the model-theoretic geometry ($MTG$).

For $AG$ consider a system $T$ of equations $w\equiv w'$, $w, w'\in W(X)$. For $LG$ we take a set of formulas $T$ in the algebra of formulas $\Phi(X)$. For $MTG$ we proceed from an $X$-type $T$. In all these cases the set can be infinite.

Now,

$\bullet$ A set $A$ in $Hom(W(X),H)$ is definable in $AG$ (i.e., $A$ is {\it an algebraic set}) if there exists $T$ in $W(X)$ such that $T'_H=A$, where
$$
T'_H=\{\mu | \ T\subset Ker(\mu) \}
$$

\medskip

$\bullet$ A set $A$ in $Hom(W(X),H)$ is definable in $LG$ (i.e., {\it $A$ is LG-definable}) if there exists $T$ in $\Phi(X)$ such that $T^L_H=A$, where
$$
T^L_H=\{\mu | \ T\subset LKer(\mu) \}=\bigcap_{u \in T}Val^{X}_H(u).
$$

\medskip

$\bullet$ A set $A$ in $Hom(W(X),H)$ is definable in $MTG$ (i.e., {\it $A$ is  MT-definable}) if there exists an $X$-type $T$  such that $T^{L_0}_H=A$, where
$$
T^{L_0}_H=\{\mu | \ T\subset Tp^H(\mu) \}=\bigcap_{u \in T}Val^{X_0}_H(u).
$$

Besides that, we  have three closures: $T''_H$ for $AG$, $T^{LL}_H$ for $LG$, and $T^{L_0L_0}_H$ for $MTG$. In the reverse direction the Galois correspondence for each of three cases above is as follows

 $$T=A'_H = \bigcap_{\mu \in A}Ker\mu).$$

 $$T=A^{L}_H = \bigcap_{\mu \in A}LKer\mu).$$

 $$T=A^{L_0}_H = \bigcap_{\mu \in A}Tp^H(\mu).$$

Correspondingly, we distinguish three types of equivalence relations on algebras from the variety $\Theta$.

\noindent
Algebras $H_1$ and $H_2$ are {\it algebraically equivalent} if
$$
T''_{H_1}=T''_{H_2}.
$$
Algebras $H_1$ and $H_2$ are {\it logically equivalent} if
$$
T^{LL}_{H_1}=T^{LL}_{H_2}.
$$
Algebras $H_1$ and $H_2$ are {\it $MT$-equivalent} if
$$
T^{L_0L_0}_{H_1}=T^{L_0L_0}_{H_2}.
$$

A natural question is

\begin{problem}\label{pr:df} {
Whether the notions of LG-definable and MT-definable sets coincide? }
\end{problem}


First we need to clarify some details. Take a special morphism $s:W(X^0)\to W(X)$ identical on the set $X \subset X^0$, $X \in \Gamma$. We have also $s_\ast:\Phi(X^0) \to \Phi(X)$. Define a set of formulas $s_\ast T = \{ s_\ast u | u \in T\}$.

\begin{theorem}\label{thm:lgmt} {
The equality $T^{L_0}_H = (s_\ast T)^L_H$ holds for for every $X$-type $T$.}
\end{theorem}



\begin{proof}
Let $\mu \in T^{L_0}_H$. Then $T \subset T^H_P(\mu)$ and every formula $u \in T$ is contained in $T^H_P(\mu)$. Besides, $s_\ast u \in LKer(\mu)$ and $\mu \in Val^X_H(u)$.  We have  $\mu \in \bigcap_{u \in T}Val^X_H(u) = (s_\ast T)^L_H$.

Let now $\mu \in  (s_\ast T)^L_H$. Then for every $u \in T$ we have $\mu \in Val^X_H(s_\ast u)$ and $s_\ast u \in LKer(\mu)$. Hence, $u \in Tp^H (\mu)$. This gives $T \subset T^H_P(\mu)$ and $\mu \in T^{L_0}_H$.
\end{proof}


Moreover, the next theorem answers Problem \ref{pr:df} in affirmative.

\begin{theorem}\label{thm:eq} Let $A\subset Hom(W(X),H)$. The set $A$ is $LG$-definable if and only if $A$ is $MT$-definable.
\end{theorem}

\begin{proof}
Theorem \ref{thm:lgmt} implies that every $MT$-definable set is $LG$-definable. Prove the opposite.

 We  use Theorem \ref{thm:zhi}: 
 for every formula $u\in \Phi(X)$ there exists an $X$-special formula $\widetilde u \in \Phi(X^0)$ such that a point $\mu: W(X)\to H$ satisfies $\widetilde u$ if and only if it satisfies $u$. Let now the set $T^{L}_H=A$ be given. Every point $\mu$ from $A$ satisfies every formula $u\in T$. Given $T$ take $T'$ consisting of all $\widetilde u$ which correspond $u\in T$. The points $\mu\in A$ satisfy every formula from $T'$. This means that $T'$ is a consistent set of $X$-special formulas. Thus $T'$ is an $X$-type, such that $A\subset T'^{L_0}_H$.

 Let now the point $\nu$ lies in  $T'^{L_0}_H$. Then $\nu$ satisfies every formula $\widetilde u$. Hence it satisfies every formula $u\in T$. Thus $\nu$ lies in $T^L_H=A$. This means that
  $$
 T'^{L_0}_H =A
 $$
 and the theorem is proved.

\end{proof}





Consider now the case when algebra $H$ is logically homogeneous and $A$ is an $Aut(H)$-orbit over the point $\mu:W(X) \to H$. We have $A=(LKer(\mu))^L_H$. The equality $LKer(\mu) = LKer(\nu)$ holds if and only if a point $\nu$ belongs to $A$. The same condition is needed for the equality $Tp^H(\mu) = Tp^H(\nu)$. Now, $\nu \in (Tp^H(\mu))^{L_0}_H$ by the definition of  $L_0$. Thus, $A=(Tp^H(\mu))^{L_0}_H$. We proved that the orbit $A$ is MT-definable and LG-definable.

Recall that we defined two full sub-categories $K_\Theta(H)$ and $LK_\Theta(H)$ in the category $Set_\Theta(H)$. Let us take one more sub-category  denoted  by $L_0 K_\Theta(H)$. In each object $(X,A)$ of this category the set $A$ is an $X$-MT-type definable set. The category $L_0 K_\Theta(H)$ is a full subcategory in $L K_\Theta(H)$. In view of Theorem \ref{thm:eq} categories $LK_\Theta(H)$ and $L_0 K_\Theta(H)$ coincide.



\begin{thebibliography}{99}


%



\bibitem{BBL}
A. Belov-Kanel, A. Berzins and R. Lipynski. Automorphisms of the
semigroup of endomorphisms of free associative algebras, {\it J.
Algebra Comput}. 17 (2007), no. 5-6, 923--939.


\bibitem{Berzins_GeomEquiv}
A.~Berzins {\it Geometrical equivalence of algebras}, Internat. J.
Algebra Comput. 11 (2001), no. 4, 447--456.

\bibitem{Berzins_2007}
A. Berzins, The group of automorphisms of the semigroup of endomorphisms of
free commutative and free associative algebras, Internat. J. Algebra Comput. 17(5-6)
(2007) 941--949.

\bibitem{BPP}
A.~Berzins, B.~Plotkin, E.~Plotkin, Algebraic geometry in
varieties of algebras with the given algebra of constants, {\it
Journal of Math. Sciences}, {\bf 102:3} (2000) 4039--4070.


\bibitem{CK}
C.C.Chang, H.J.Keisler, Model theory,North Holland, 1977, 554pp.







\bibitem{Formanek}
E. Formanek, A question of B. Plotkin about the semigroup of endomorphisms of a
free group, {\it Proc. Amer. Math. Soc.} {\bf 30(4)} (2002) 935--937.


\bibitem{Halmos}
P.R.~Halmos, {\it Algebraic logic}, New York, 1969.

\bibitem{Hodges} W.Hodges,
{\it A Shorter Model theory}, Cambridge University Press.








\bibitem{KLP}
Y.~Katsov, R.~Lipyanski, B.~Plotkin, Automorphisms of the
categories of free modules, free semimodules and free Lie
algebras, {\it Comm. Algebra}, {\bf 35} (2007) 931--952.

\bibitem{L}
R.~Lipyanski, Automorphisms of the endomorphism semigroups of free linear algebras
of homogeneous varieties, {\it Linear Algebra Appl.} 429(1) (2008) 156--180.

\bibitem{LP}
R.~Lipyanski, B.~Plotkin, Automorphisms of categories of free modules and free
Lie algebras, arXiv:math.RA/0502212, [math.GM], 2005.

\bibitem{Marker}
D.~Marker, {\it Model Theory: An Introduction} (Springer Verlag, 2002).



\bibitem{MPP1}
G.~Mashevitzky, B.~Plotkin, E.~Plotkin, Automorphisms of the category of free Lie algebras,
{\it J. Algebra}, {\bf 282:2} (2004) 490--512.


\bibitem{MPP2}
G.~Mashevitzky, B.~Plotkin, E.~Plotkin, Automorphisms of categories of free algebras of
varieties, {\it Electronic Research Announcements of AMS}, {\bf 8} (2002) 1--10.

\bibitem{MS}
G.~Mashevitzky, B.M.~Schein, Automorphisms of the endomorphism semigroup of
a free monoid or a free semigroup, {\it Proc. Amer. Math. Soc.} {\bf 131:(6)} (2003) 1655--1660.



\bibitem{MSZ}
G.~Mashevitzky, B.M.~Schein and G.I.~Zhitomirski, Automorphisms of the semigroup
of endomorphisms of free inverse semigroups, {\it Comm. Algebra} {\bf 34:10} (2006) 3569--3584.

\bibitem{MR}
A.~Myasnikov, V.~Remeslennikov,  Algebraic geometry over groups
II, Logical foundations, {\it J. of Algebra} {  234} 
(2000) 225--276.


\bibitem{Houcine}
A.~Ould Houcine, Homogenity and prime models in torsion-free hyperbolic groups, {\it
Confluentes Mathematici} {  3} (2011) 121--155.

\bibitem{PerinSklinos}
C.~Perin, R.~Sklinos, Homogenety in the free group, {\it Duke Mathematical Journal}, {  161(13)} (2012), pp. 2635 -- 2658.


\bibitem{Pillay}
A. ~Pillay, On genericity and weight in the free group, {\it Proc. Amer. Math. Soc.} {  137}
(2009), 3911--3917.

\bibitem{Plotkin_UA-AL-Datab}
B.~Plotkin  {\it Universal algebra,algebraic logic and databases.} Kruwer Acad. Publ.,
(1994).

\bibitem{Plotkin_Haz}
B.~Plotkin  {\it Algebra, categories, and databases,} Handbook of Algebra, {  v.2},  Elsevier (2000), 81--148.
(1994).

\bibitem{Plotkin_7_lec}
B.~Plotkin, Seven lectures on the universal algebraic geometry,
Preprint,(2002), Arxiv:math, GM/0204245, 87pp.

\bibitem{Plotkin_VarAlg-AlgVar}
B.~Plotkin, Varieties of algebras and algebraic varieties, Israel
J. Math., {  96(2)} (1996), 511--522.

\bibitem{Pl-Sib}
B.~Plotkin,
 {\it Varieties of algebras and algebraic varieties. Categories of
       algebraic varieties.} Siberian Advanced Mathematics,
       Allerton Press, v.7, issue 2, 1997, p.64-97.


 \bibitem{Pl-Zapiski}
B.~Plotkin,
{\it Geometrical equivalence, geometrical
similarity, and geometrical compatibility of algebras}, Zapiski
Nauch. Sem. POMI, v.330, (2006), 201 -- 222.


\bibitem{Pl-St}
B.~Plotkin, {\it Algebras with the same algebraic geometry},
Proceedings of the International Conference on Mathematical
 Logic, Algebra and Set Theory,  dedicated to 100 anniversary
of P.S.Novikov, Proceedings of the Steklov Institute of
Mathematics, MIAN, {  v.242}, (2003), 176 -- 207. Arxiv:
math.GM/0210194.

\bibitem{Plotkin_AGinFOL} 
B.~Plotkin, {\it Algebraic geometry in First Order Logic}, Sovremennaja Matematika and
Applications {  22} (2004), 16--62. Journal of Math. Sciences, {  137}, n.5, (2006),
5049-- 5097. 



\bibitem{Plotkin_Gagta} B.~Plotkin, {\it
Algebraic logic and logical geometry in arbitrary varieties of algebras}, In: Proceedings of the Conf. on Group Theory, Combinatorics and Computing, AMS Contemporary Math. series, (2014) 151--169.

\bibitem{Plotkin_SomeResultsUAG}
B.~Plotkin, {\it Some results and problems related to universal algebraic geometry,}
International Journal of Algebra and Computation,  {  17(5/6)} (2007), 1133-1164.


\bibitem{PlAlPl}
B.~Plotkin, E.~Aladova, E.~Plotkin,  {\it Algebraic logic and logically-geometric types in varieties of algebras,}
  Journal of Algebra and its Applications, {  12(2)},(2012), Paper No. 1250146, 23 p.

\bibitem{BPlAlEPl}
B.~Plotkin, E.~Aladova, E.~Plotkin,
Algebraic logic and logical geometry in arbitrary varieties of algebras, Monograph,  in Progress.

\bibitem{PZ}
B.~Plotkin, G.~Zhitomirski, {\it Some logical invariants of algebras and logical
relations between algebras}, Algebra and Analysis, 19:5, (2007), p. 214--245,
St. Petersburg Math. J., 19:5, (2008), p. 859--879.

\bibitem{PZ1}
B.~Plotkin , G.~Zhitomirski, On automorphisms of categories of universal algebras, {J.
Algebra Comput} {\bf 17:(5-6)} (2007) 1115--1132.

\bibitem{PZ2}
B.~Plotkin , G.~Zhitomirski, Automorphisms of categories of free algebras of some varieties, {J.
Algebra} {\bf 306:2} (2006) 344--367.





\bibitem{Sklinos_1} R.~Sklinos, Unpublished.

\bibitem{S}
X.~Sun, Automorphisms of the endomorphism semigroup of a free algebra;
 International Journal of Algebra and Computation, to appear.


\bibitem{ST}
I.~Shestakov, A.~Tsurkov, Automorphic equivalence of the representations of Lie algebras. Preprint, arXiv:1210.2660v1, 28pp.

\bibitem[Sz]{Sz}
W. Szmielew, Elementary properties of Abelian groups, Fund. Math.
41 (1955) 203-271.

\bibitem{Ts1}
A. Tsurkov, Automorphic equivalence of algebras. {\it International Journal of
Algebra and Computation}. {\bf 17:5/6}, (2007), 1263--1271.

\bibitem{Ts2}
 A.~Tsurkov. Àutomorphisms of the category of the free nilpotent groups of the fixed class
of nilpotency, {\it J. Algebra Comput.} {\bf 17:(5-6)} (2007) 1273--1281.

\bibitem{Ts3}
 A. Tsurkov, Automorphic equivalence of linear algebras,
http://arxiv.org/abs/1106.4853, 15pp.  Accepted in the Journal of Algebra
and Its Applications.

\bibitem{Ts4}
A.~Tsurkov, Automorphic Equivalence in the Classical Varieties of
Linear Algebras, Preprint, arXiv:1309.2314v3, 16pp.

\bibitem{Ts5}
A.~Tsurkov, Automorphic Equivalence of Many-Sorted Algebras,Preprint, arXiv:1304.0021v6, 39pp.

\bibitem{TP}
A.~Tsurkov, B.~Plotkin, Action type geometrical equivalence of representations of groups,
{\it Algebra Discrete Math.}, {\bf 4} (2005) 48--80.



\bibitem{Zhitom_types}
G.~Zhitomirski, On logically-geometric types of algebras, Preprint
arXiv: 1202.5417v1 [math.LO].

\end{thebibliography}
\end{document}